\documentclass{amsart}

\usepackage{amsmath, amsfonts, amssymb, amsthm, amscd, graphicx, float, epstopdf, bm, tabu
}

\allowdisplaybreaks[4]

\usepackage[pdfpagemode={UseOutlines},bookmarks=true,bookmarksopen=true, bookmarksopenlevel=0,bookmarksnumbered=true,hypertexnames=false, colorlinks,linkcolor={blue},citecolor={blue},urlcolor={blue}, pdfstartview={FitV},unicode,breaklinks=true,backref=page]{hyperref}

\newtheorem{thm}{Theorem}[section]
\newtheorem{remark}[thm]{Remark}
\newtheorem{notation}[thm]{Notation}
\newtheorem{conjecture}[thm]{Conjecture}
\newtheorem{defn}[thm]{Definition}
\newtheorem{lem}[thm]{Lemma}
\newtheorem{prop}[thm]{Proposition}
\newtheorem{cor}[thm]{Corollary}
\newtheorem{example}[thm]{Example}
\begin{document}

\title[Volume of the moduli space of bounded $\mathbb{RP}^2$ structures]{Volume of the moduli space of unmarked bounded positive convex $\mathbb{RP}^2$ structures}


\author{Zhe Sun}
\address{Department of mathematics, University of Luxembourg}
\email{zhe.sun@uni.lu}
\thanks{The author was partially supported by the FNR AFR bilateral grant COALAS 11802479-2. }



\date{}

\begin{abstract}
For the moduli space of unmarked convex $\mathbb{RP}^2$ structures on the surface $S_{g,m}$ with negative Euler characteristic, we investigate the subsets of the moduli space defined by the notions like boundedness of projective invariants, area, Gromov hyperbolicity constant, quasisymmetricity constant etc. These subsets are comparable to each other. We show that the Goldman symplectic volume of the subset with certain projective invariants bounded above by $t$ and fixed boundary simple root lengths $\mathbf{L}$ is bounded above by a positive polynomial of $(t,\mathbf{L})$ and thus the volume of all the other subsets are finite. We show that the analog of Mumford's compactness theorem holds for the area bounded subset. 
\end{abstract}

\maketitle


\section{Introduction}
In \cite{Mir07a}, Mirzakhani showed that the volume of the moduli space $\mathcal{M}_{g,m}(\mathbf{L})$ of Riemann surfaces with fixed boundary lengths $\mathbf{L}$ with respect to the Weil--Petersson symplectic form is a polynomial of $\mathbf{L}$. She obtained this result by showing a beautiful recursive formula where one side consists of the volume of $\mathcal{M}_{g,m}(\mathbf{L})$, while the other side consists of the volumes of the moduli spaces of Riemann surfaces that cutting out a pair of pants from $S_{g,m}$ (see \cite{Wri19} for a survey). The higher Teichm\"uller theory studies the representations of the fundamental group $\pi_1(S_{g,m})$ with more flexibility where the isometry group $\operatorname{PSL}(2,\mathbb{R})$ of the holonomy representation of the hyperbolic surface is replaced by a semisimple Lie group (see \cite{W19} for a survey). We are looking for an analog of Mirzakhani's result for the special connected component of the {\em $\rm{PGL}(n,\mathbb{R})$-representation variety} $\rm{Hom}(\pi_1(S_{g,m}),\rm{PGL}(n,\mathbb{R}))/\rm{PGL}(n,\mathbb{R})$ modulo the mapping class group. The existence of such geometric quantity was predicated by Labourie and McShane in \cite[page 284]{LM09}, and they indicate that the volume is not the right quantity to compute since it is infinite for $n\geq 3$. In this paper, we work on the existence of such geometric quantity for $n=3$ and we propose several finite quantities which are comparable in sense of coarse geometry.

A {\em convex $\mathbb{RP}^2$ surface} is a quotient $\Omega/\Gamma$ where $\Omega\subset \mathbb{RP}^2$ is convex and $\Gamma\subset \operatorname{PGL}(3,\mathbb{R})$ is discrete and acting properly on $\Omega$. It was initially studied by Kuiper \cite{Ku53,Ku54}, Benz\'ecri \cite{B60}, Kac--Vinberg \cite{KV67} and many others. On the other hand, a special connected component $\rm{Hit}_n(S_{g,0})$ of the $\rm{PGL}(n,\mathbb{R})$-representation variety was found by Hitchin in \cite{Hit92} through a special section of the Hitchin fibration. In \cite{CG93,G90}, Goldman--Choi proved that the moduli space of convex $\mathbb{RP}^2$ structures on $S_{g,0}$ is exactly $\rm{Hit}_3(S_{g,0})$. In \cite{FG06}, Fock and Goncharov introduced the notion of positivity to study a special part $\rm{Pos}_n(S_{g,m})$ of the representation variety $\rm{Hom}(\pi_1(S_{g,m}),\rm{PGL}(n,\mathbb{R}))/\rm{PGL}(n,\mathbb{R})$. By \cite[Theorem 1.15]{FG06}, $\rm{Pos}_n(S_{g,0})=\rm{Hit}_n(S_{g,0})$. For $n=3$, positivity can be understood as $\partial \Omega$ partly strictly convexity (Definition \ref{defn:pos3}). Let $\rm{Pos}_3(S_{g,m})(\mathbf{L})$ be the positive representation variety with fixed boundary simple root lengths $\mathbf{L}$. For $m>0$, in \cite{Mar10}, Marquis proved that the moduli space of cusped strictly convex $\mathbb{RP}^2$ structures on $S_{g,m}$ is exactly $\rm{Pos}_3(S_{g,m})(\mathbf{0})$. For $\mathbf{L}\in \mathbb{R}_{>0}^{2m}$, by \cite[Section 9]{LM09}, we can double the representation by doubling the surface in a canonical way, thus we identify $\rm{Pos}_3(S_{g,m})(\mathbf{L})$ with the moduli space of the resulting doubled convex $\mathbb{RP}^2$ structures. (For the other cases, the convex $\mathbb{RP}^2$ structure for the positive representation is investigated in \cite{Mar12}.) Hence we call $\mathcal{H}(S_{g,m})(\mathbf{L}):=\rm{Pos}_3(S_{g,m})(\mathbf{L})/Mod(S_{g,m})$ the {\em moduli space of unmarked positive convex $\mathbb{RP}^2$ structures on $S_{g,m}$} with fixed boundary simple root lengths $\mathbf{L}$.

The {\em (Atiyah--Bott--)Goldman symplectic form} \cite{AB83,G84} is a nature mapping class group invariant symplectic form on $\rm{Pos}_3(S_{g,m})(\mathbf{L})$ which generalizes the Weil--Petersson symplectic form. As pointed out by Labourie and McShane \cite{LM09}, the Goldman symplectic volume of $\mathcal{H}(S_{g,m})(\mathbf{L})$ is infinite. To get a finite number, we suggest to integrate over a subset of $\mathcal{H}(S_{g,m})(\mathbf{L})$ or integrate another function over $\mathcal{H}(S_{g,m})(\mathbf{L})$. We are mainly interested in the following two candidates:
\begin{enumerate}
\item $\mathcal{H}^t(S_{g,m})(\mathbf{L})$ which is a subset of $\mathcal{H}(S_{g,m})(\mathbf{L})$ with extra projective invariants (Definition \ref{definition:bpos}) comparing to the $3$-Fuchsian representations bounded above by $t$, and
\item $\mathcal{AH}^t(S_{g,m})(\mathbf{L})$ which is the subset with the canonical area (Definition \ref{definition:carea} which generalizes the hyperbolic area) bounded above by $t$. We suggest that $\mathcal{AH}^t(S_{g,m})(\mathbf{L})$ is the most natural subset to consider since for each element in the moduli space $\mathcal{M}_{g,m}(\mathbf{L})$, the hyperbolic area is a fixed constant.
\end{enumerate}
Inspired by the work of Benoist \cite{Ben03} and Colbois--Vernicos--Verovic \cite{CVV08}, we introduce many other subsets defined with respect to the structure constants, like hyperbolicity constant $B^t(S_{g,m})(\mathbf{L})$, quasisymmetricity constant $C^t(S_{g,m})(\mathbf{L})$, harmonicity constant $F^t(S_{g,m})(\mathbf{L})$ etc (Example \ref{example:exhaustion}). For $\mathbf{L}\in \mathbb{R}_{>0}^{2m}$, we proved that these subsets are comparable to each other, which allows us to prove the volume finiteness for all the above mentioned subsets by proving the volume finiteness for $\mathcal{H}^t(S_{g,m})(\mathbf{L})$, particularly we obtain the volume finiteness of $\mathcal{AH}^t(S_{g,m})(\mathbf{L})$. 


\begin{thm}[Main Theorem \ref{thm:main}]
For $\mathbf{L}\in \mathbb{R}_{>0}^{2m}$, the Goldman symplectic volume of $\mathcal{H}^t(S_{g,m})(\mathbf{L})$ is bounded above by a positive polynomial of $(t,\mathbf{L})$.
\end{thm}
Notice that the union $\cup_{t>0} \mathcal{H}^t(S_{g,m})(\mathbf{L})$ provides an exhaustion of $\mathcal{H}(S_{g,m})(\mathbf{L})$.
\begin{cor}
For $\mathbf{L}\in \mathbb{R}_{>0}^{2m}$, the Goldman symplectic volume of $\int_{\mathcal{H}(S_{g,m})(\mathbf{L})} 
e^{-t} dVol$ is finite where $t$ is defined to be the minimal value such that $\rho \in \mathcal{H}^t(S_{g,m})(\mathbf{L})$.
\end{cor}
There are two crucial tools used by Mirzakhani \cite{Mir07a} for integrating over the moduli space $\mathcal{M}_{g,m}(\mathbf{L})$:
\begin{enumerate}
\item Wolpert's Magic Formula \cite{Wol82,Wol83} which expresses the Weil--Petersson symplectic form in terms the Fenchel--Nielsen coordinates with respect to a pants decomposition $\mathcal{P}$ and a choice of transverse arcs to $\mathcal{P}$;
\item McShane's identity \cite{McS98} and generalized McShane's identity for the hyperbolic surface with geodesic boundary in \cite[Theorem 4.2]{Mir07a}.
\end{enumerate}

We prove our main theorem by adopting the Mizakhani's proof in \cite{Mir07a} for $\mathcal{H}_3(S_{g,m})(\mathbf{L})$ except where we estimate. The original Mirzakhani's Integration formula can be naturally extended to Theorem \ref{theorem:mirint}. Similarly, we have two corresponding crucial tools:
\begin{enumerate}
\item (Theorem \ref{theorem:Darboux}) generalized Wolpert's Magic Formula provided by Sun--Wienhard--Zhang \cite{SWZ17,SZ17} with respect to an ideal triangulation $\mathcal{T}$ subordinate to a pants decomposition $\mathcal{P}$ and a choice of transverse arcs to $\mathcal{P}$;
\item (Theorem \ref{theorem:mcid}) generalized McShane's identity provided by Huang--Sun \cite{HS19} for each simple root length of the boundary component, which is expressed similarly to McShane--Mirzakhani identity.
\end{enumerate}
Then we use the definition of $\mathcal{H}^t(S_{g,m})(\mathbf{L})$ and Lemma \ref{lemma:pl} to estimate, which show that the existence of such geometric quantity.

By \cite{Z15}, the Mumford compactness theorem fails on the entire space $\mathcal{H}(S_{g,m})(\mathbf{L})$.
We will show the analog of Mumford compactness theorem for the area bounded subset $\mathcal{AH}^{t}(S_{g,m})(\mathbf{L})$. Let $\mathcal{AH}^{t}(S_{g,m})(\mathbf{L})_\epsilon$ be the subset of $\mathcal{AH}^{t}(S_{g,m})(\mathbf{L})$ with the simple root length systoles are bigger or equal to $\epsilon>0$.  
\begin{thm}[Theorem \ref{theorem:mum}]
The subset $\mathcal{AH}^{t}(S_{g,m})(\mathbf{L})_\epsilon$ with $\mathbf{L}\in \mathbb{R}_{>0}^{2m}$ is compact.
\end{thm}


We would like to ask the following two questions as a first step for the further investigation:
\begin{enumerate}
\item For $\mathcal{H}^t(S_{g,m})(\mathbf{L})$, how the minimal value $t$ such that $\rho \in \mathcal{H}^t(S_{g,m})(\mathbf{L})$ varies in $\mathcal{H}_3(S_{g,m})(\mathbf{L})$ with respect to the Fock--Goncharov parameters subordinate to a pants decomposition that we use? 
\item For $\mathcal{AH}^t(S_{g,m})(\mathbf{L})$, how to express the canonical area in term of the Fock--Goncharov parameters (even for one pair of pants)? 
\end{enumerate}
There are several approaches to get some geometric quantities for the moduli space $\mathcal{H}(S_{g,m})(\mathbf{L})$:
\begin{enumerate}
\item We can try to find both the lower and upper bound of the Goldman symplectic volume of $\mathcal{H}^t(S_{g,m})(\mathbf{L})$($\mathcal{AH}^t(S_{g,m})(\mathbf{L})$ resp.) sharp enough such that we can compute the top term of its expansion in $t$.
\item In \cite{W18}, Wienhard suggested to divide $\rm{Pos}_3(S_{g,m})(\mathbf{L})$ by a larger group than the mapping class group which preserves the Goldman symplectic form such that the volume of the quotient is finite. In this approach, it is not clear if the cluster mapping class group in \cite[page 29]{FG06} works.
\item In \cite{Mir07b}, Mirzakhani demonstrated the link between the volumes and the intersection theory on the moduli space of curves which allowed her to again prove Witten--Kontsevich theorem \cite{Ko92,Wi91}. We expect an intersection theory for $\mathcal{H}(S_{g,m})(\mathbf{L})$.
\end{enumerate}

\section{Convex $\mathbb{RP}^2$ structures on surfaces} 
\label{section:2}
We recall some preliminaries for investigating the moduli space of unmarked convex $\mathbb{RP}^2$ structures on the surface, including the convex $\mathbb{RP}^2$ structures on surfaces, the positive representations and the projective invariants that are used to parameterize the moduli space.

\subsection{Convex $\mathbb{RP}^2$ structure}
Let $S=S_{g,m}$ be a smooth surface of genus $g$ and $m$ holes with negative Euler characteristic.
\begin{defn}[$\mathbb{RP}^2$ surface]
The $\mathbb{RP}^2$ surface $\Sigma$ is a quotient $\Omega/\Gamma$ diffeomorphic to a smooth surface $S$, where $\Omega$ is a convex domain in $\mathbb{RP}^2$ and $\Gamma$ is a discrete subgroup of $\rm{PGL}(3,\mathbb{R})$ acting properly on $\Omega$.

Two $\mathbb{RP}^2$ surfaces $\Omega/\Gamma$ and $\Omega'/\Gamma'$ are equivalent if there is a projective transformation $g\in \rm{PGL}(3,\mathbb{R})$ such that $(\Omega',\Gamma')=(g  \Omega, g  \Gamma   g^{-1})$.
\end{defn}

The $\mathbb{RP}^2$ surface $\Sigma$ is equivalent to a pair $(\rho,f)$:
\begin{itemize}
\item $\rho:\pi_1(S)\rightarrow \operatorname{PGL}(3,\mathbb{R})$ is the holonomy representation of $\Sigma$ where $\rho(\pi_1(S))=\Gamma$;
\item $f:\widetilde{S}\rightarrow\mathbb{RP}^2$ is the developing map where $f(\widetilde{S})=\Omega$. 
\end{itemize}
The shape of the domain $\Omega$ is an important feature for the $\mathbb{RP}^2$ surface. 
\begin{defn}
\begin{enumerate}
\item A subset $\Omega$ in $\mathbb{RP}^2$ is {\em convex} if the intersection of $\Omega$ with every line is connected.
\item The convex subset $\Omega$ is {\em properly convex} if $\Omega$ is contained in $\mathbb{R}^2 \cong \mathbb{RP}^2 \backslash \mathbb{RP}^1$ for some hyperplane $\mathbb{RP}^1$.
\item The properly convex subset $\Omega$ is {\em strictly convex} if the boundary $\partial \Omega$ contains no line segments. 
\end{enumerate}
\end{defn}

\begin{defn}[Convex $\mathbb{RP}^2$ structure on $S$]
A {\em (marked) convex $\mathbb{RP}^2$ structure} on a smooth surface $S$ is defined to be a diffeomorphism $h: S\rightarrow \Sigma$ where $\Sigma$ is a convex $\mathbb{RP}^2$ surface.

We say that two (marked) convex $\mathbb{RP}^2$ structures $(h,\Sigma)$ and $(h',\Sigma')$ are equivalent if and only if there is a projective equivalence $g: \Sigma\rightarrow \Sigma'$ such that $g\circ h$ is isotopic to $h'$.

The {\em unmarked convex $\mathbb{RP}^2$ structure} on the smooth surface $S$ is the (pure) mapping class group orbit of a marked convex $\mathbb{RP}^2$ structure. 

We say that two unmarked convex $\mathbb{RP}^2$ structures $[h,\Sigma]$ and $[h',\Sigma']$ are equivalent if and only if there is a projective equivalence $g:\Sigma \rightarrow \Sigma'$ and an orientation preserving diffeomorphism $u$ of $S$ which fixes the boundary such that $g\circ h \circ u$ is isotopic to $h'$.
\end{defn}
The unmarked convex $\mathbb{RP}^2$ structure is the mapping class group orbit of marked convex $\mathbb{RP}^2$ structures. There is a natural (Finsler) metric defined for any convex domain.
\begin{defn}[Hilbert metric]
Given a convex domain $\Omega \subset \mathbb{R}^2\subset \mathbb{RP}^2$, for any two distinct points $a,b \in \Omega$, let $p_a$ and $p_b$ be the points at which the straight line $ab$ intersects the boundary of $\Omega$, where $p_a$ is closer to $a$ and $p_b$ is closer to $b$. Let $|\cdot|$ be the Euclidean length in $\mathbb{R}^2$. The {\em Hilbert distance} is defined to be
\begin{equation*}
d_\Omega(a,b)=\frac{1}{2}\log \left(\frac{|a-p_b|}{|b-p_b|} \cdot \frac{|b-p_a|}{|a-p_a|}\right).
\end{equation*}
The metric defined by the Hilbert distance is called the {\em Hilbert metric}. 
The Hilbert distance is invariant under projective transformations. Thus for a convex $\mathbb{RP}^2$ surface $\Omega/\Gamma$, the Hilbert metric on $\Omega$ descends to the {\em Hilbert metric} on $\Omega/\Gamma$. 
\end{defn}
In the special case when $\Omega$ is an ellipse for the convex $\mathbb{RP}^2$ surface $\Omega/\Gamma$. Then the Hilbert metric on is $\Omega$ is the usual hyperbolic metric on $\Omega$ with respect to the Klein model.

\begin{defn}[Area]
For any $(x,v) \in T \Omega$ where $x$ belongs to the convex domain $\Omega$ and $v$ is the tangent vector in $\mathbb{R}^2$, we note $x^+$ ($x^-$ resp.) the intersection points of the boundary $\partial \Omega$ and the ray defined by $x$ and $v$ ($-v$ resp.). We define
\[|v|_x = \frac{d}{dt}|_{t=0} d_\Omega(x,x+tv)=\frac{1}{2}\left(\frac{1}{|x-x^-|}+\frac{1}{|x-x^+|} \right) |v|.\]
\begin{itemize}
\item Let $B_x(1)=\{v \in T_x\Omega\; |\; |v|_x<1\}$.
\item Let $E_B=\pi$ be the Euclidean volume
of the open unit ball in $\mathbb{R}^2$.
\item Let $\rm{Leb}$ be the canonical Lebesgue measure of $\mathbb{R}^2$ equal to $1$ on the unit square.
\item The density is $h_\Omega(x):= \frac{E_B}{\rm{Leb}(B_x(1))}$.
\end{itemize}
For any Borel set $A$ of $\Omega$, the {\em area} of $A$ is defined with respect to the Busemann measure:
\[\rm{Vol}_{\Omega}(A)=\int_A h_\Omega(x) d\rm{Leb}(x).\]
\end{defn}
\begin{remark}
\label{remark:blaschke}
There are many other areas defined with respect to different proper densities \cite{V13}. By a co-compactness result of Benz\'ecri \cite{B60}, any pair of proper densities are comparable. Notably, there is the Blaschke metric which is Riemannian and uniformly comparable to the Hilbert metric \cite[Proposition 3.4]{BH13}. 
\end{remark}

\subsection{Positive representations}
In this subsection, let $S=S_{g,m}$ be a topological surface of genus $g$ and $m$ holes with negative Euler characteristic. We study the convex $\mathbb{RP}^2$ structure on $S$ from representation theory point of view.  The holonomy representations of the $\mathbb{RP}^2$ surfaces are contained in $\rm{Hom}(\pi_1(S),\rm{PGL}(3,\mathbb{R}))$.
Modulo the equivalence relation, the {\em representation variety} for $\rm{PGL}(3,\mathbb{R})$ is 
\[\rm{Hom}(
\pi_1(S),\rm{PGL}(3,\mathbb{R}))/\rm{PGL}(3,\mathbb{R})\] where $\rm{PGL}(3,\mathbb{R})$ acts by conjugation. When the holonomy representation is nice enough, there is a one-to-one correspondence between the convex $\mathbb{RP}^2$ structure on $S$ up to equivalence and its holonomy representation up to conjugation. 

The {\em $3$-Fuchsian representation} is the composition of the discrete faithful representation from $\pi_1(S)$ to $\rm{PSL}(2,\mathbb{R})$ and the irreducible representation $\rm{PSL}(2,\mathbb{R})$ to $\rm{PGL}(3,\mathbb{R})$. 
\begin{defn}\cite[Hitchin component]{Hit92}
\label{definition:Hit}
For $S=S_{g,0}$ being a closed surface of genus $g\geq 2$, the $\rm{PGL}(3,\mathbb{R})$-{\em Hitchin component} $\rm{Hit}_3(S)$ is the connected component of $\rm{Hom}(
\pi_1(S),\rm{PGL}(3,\mathbb{R}))/\rm{PGL}(3,\mathbb{R})$ that contains all the deformations of $3$-Fuchsian representations.
\end{defn}
\begin{thm}\cite{CG93,G90}
\label{thm:11cor}
For the integer $g\geq 2$, the moduli space of marked strictly convex $\mathbb{RP}^2$ structures on the surface $S_{g,0}$ is homeomorphic to $\rm{Hit}_3(S_{g,0})$, which is a cell of dimension $16g-16$.  
\end{thm}

For general $n$, the geometric features of the Hitchin component were unravelled by Fock and Goncharov\cite{FG06} using positivity and independently by Labourie \cite{Lab06} using Anosov flows. Thus the notion of Hitchin representation was generalized to {\em positive representation} and {\em Anosov representation} in two directions. Both the positive representations and the Anosov representations are proved to be discrete and faithful.

We focus on the positive representations in this paper. Let us recall the definition of the positive representations.

\begin{defn}[Flags]
\label{defn:flag}
A \emph{flag} $F$ in $\mathbb{R}^3$ is a maximal filtration of vector subspaces of $\mathbb{R}^3$:
\[
\{0\}=F^{(0)} \subset F^{(1)} \subset F^{(2)} \subset F^{(3)}=\mathbb{R}^3,\quad \dim F^{(1)}=i,
\]
denoted by $(F^{(1)},F^{(2)})$. The {\em flag variety} is denoted by $\mathcal{B}$.
Usually, we consider the flag $(F^{(1)},F^{(2)})$ as $(x,X)$ where $x\in \mathbb{RP}^2$ and $X$ is a line crossing $x$ in $\mathbb{RP}^2$.

A \emph{basis} for a flag $F=(F^{(1)},F^{(2)})$ is a basis $(f_1,f_2,f_3)$ for the vector space $\mathbb{R}^3$ such that the first $i$ vectors form a basis for $F^{(i)}$, for $i=1,2$.

\end{defn}

\begin{defn}[Generic position]
We say that the (ordered) $d$-tuple of flags $(F_1,\cdots,F_d)$ are in {generic position} if for any integers $1\leq a<b<c\leq d$ and non-negative integers $i_a,i_b,i_c$  with $i_a+i_b+i_c\leq 3$,
the sum \[F_{a}^{(i_a)}+F_{b}^{(i_b)}+F_{c}^{(i_c)}\]
is direct.
\end{defn}

\begin{defn} \cite[Lemma 9.7]{FG06}
\label{defn:pos3}
For the integer $d\geq 3$, we say that the $d$-tuple of generic flags $(F_1,\cdots,F_d)$ in $\mathbb{RP}^2$ is positive if and only if there exists a strictly convex curve (that bounds a strictly convex domain) such that the curve is passing the points $(F_1^{(1)},\cdots,F_d^{(1)})$ with respect to the cyclic order and is tangent to the lines $(F_1^{(2)},\cdots,F_d^{(2)})$ (see Figure \ref{figure:positiveconf}). We define $\rm{Conf}^+_d$ to be the space of positive $d$-tuples of flags up to diagonal projective transformations.

\begin{figure}[ht]
\centering
\includegraphics[scale=0.35]{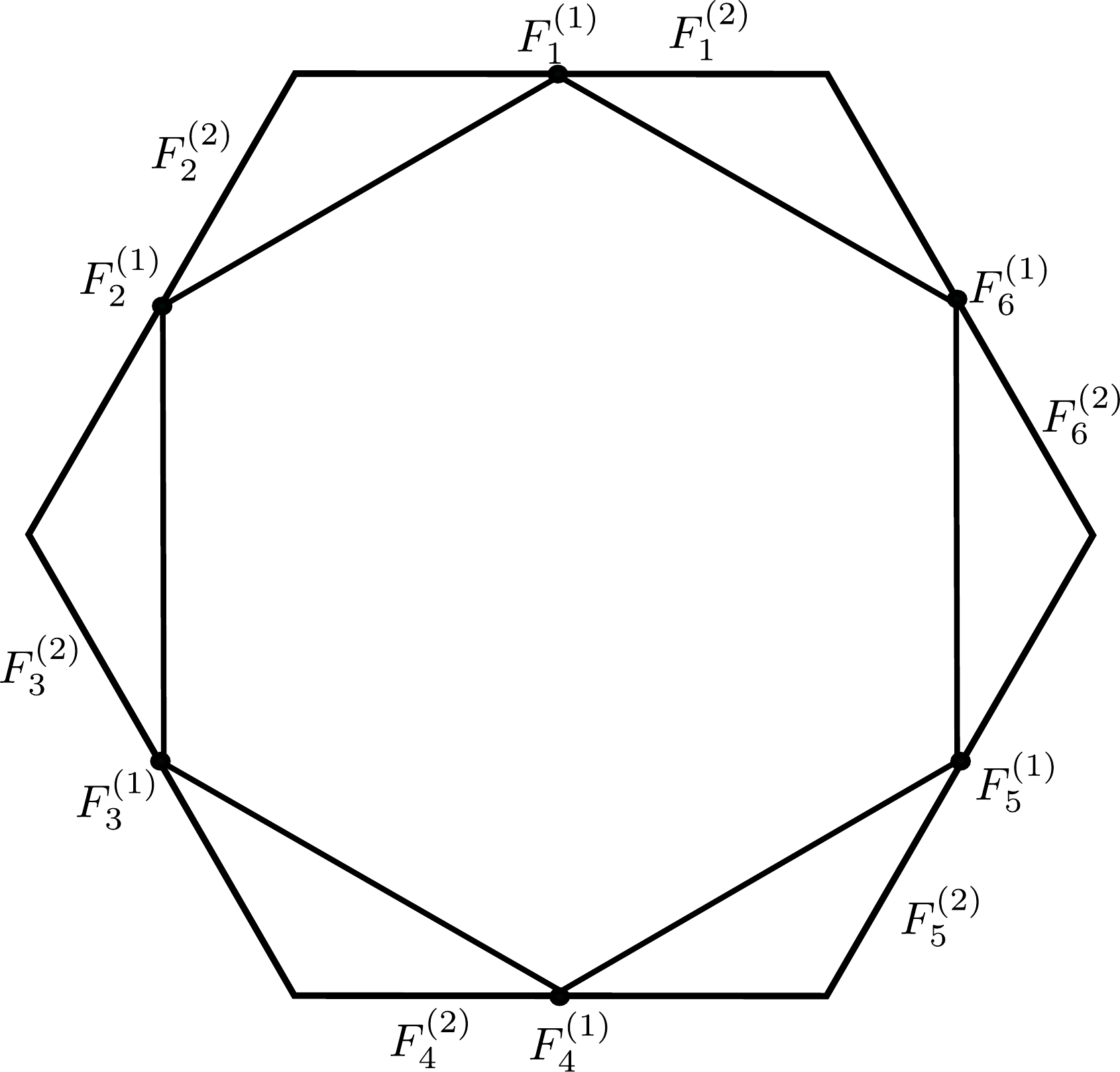}
\small
\caption{A positive $6$-tuple of flags.}
\label{figure:positiveconf}
\end{figure}

For any subset $C$ of a circle, we say that the continuous map $\xi:C\rightarrow \mathcal{B}$ is {\em positive} if for any cyclically ordered set $(x_1,\cdots,x_d)$ of $C$ with $d\geq 3$, $(\xi(x_1),\cdots,\xi(x_d))$ is a positive $d$-tuple of flags.
\end{defn}

\begin{defn}[Boundary at infinity]
Let $S_{g,m}$ be a topological surface with negative Euler characteristic. For each $\rho$ belongs to $\rm{Hom}(
\pi_1(S_{g,m}),\rm{PGL}(3,\mathbb{R}))/\rm{PGL}(3,\mathbb{R})$, we choose an auxiliary complete hyperbolic structure $\rho_h$ with geodesic boundary:
\begin{enumerate}
\item for each boundary component $\alpha$, if the monodromy $\rho(\alpha)$ is unipotent, we choose $\rho_h$ such that the boundary $\alpha$ is a cusp;
\item for each boundary component $\alpha$, if the monodromy $\rho(\alpha)$ is not unipotent, we choose $\rho_h$ such that the length of $\alpha$ with respect to $\rho_h$ is not zero. 
\end{enumerate}
Let $(\widetilde{S_{g,m}},\rho_h)$ be the universal cover of $(S_{g,m},\rho_h)$. The {\em boundary at infinity} $\partial_\infty \pi_1(S_{g,m})$ is the intersection of the absolute $\partial \mathbb{H}^2$ with the closure of $(\widetilde{S_{g,m}},\rho_h)$. 
\end{defn}

If $m=0$, $\partial_\infty \pi_1(S_{g,m})$ is homeomorphic to a circle. If $m\neq 0$ and each boundary $\alpha$ of $S_{g,m}$ with respect to $\rho_h$ is a cusp, the boundary at infinity $\partial_\infty \pi_1(S_{g,m})$ is homeomorphic to a circle. If $m\neq 0$ and the length of some geodesic boundary of $S_{g,m}$ with respect to $\rho_h$ is non-zero, the boundary at infinity $\partial_\infty \pi_1(S_{g,m})$ is homeomorphic to Cantor set on a circle. One can think of $\alpha^+$ and $\alpha^-$ approaching to each other when the length of $\alpha$ with respect to $\rho_h$ approaches to zero.

\begin{defn}[Positive representation]
The representation $\rho:\pi_1(S_{g,m})\rightarrow \rm{PGL}(3,\mathbb{R})$ is positive if there exists a $\rho$-equivariant map $\xi_\rho:\partial_\infty \pi_1(S_{g,m})\rightarrow \mathcal{B}$ is positive. We denote the space of positive representations by $\rm{Pos}_3(S_{g,m})$.
\end{defn}

Let us recall a nice geometric description of the positive representations. We restrict to $\rm{PGL}(3,\mathbb{R})$ case even through the following statements are true for any split semisimple algebraic group. 

\begin{thm}\cite[Theorem 2.8]{FG07}
\label{theorem:loxo}
We say an element in $\rm{PGL}(3,\mathbb{R})$ is {\em loxodromic} if it is conjugate to $diag(\lambda_1,\lambda_2,\lambda_3)$ where $\lambda_1>\lambda_2>\lambda_3>0$. A matrix is {\em totally positive} if all the minors are positive numbers. A upper triangular matrix is {\em totally positive} if all the minors are positive numbers except the ones that have to be zero due to the upper triangular condition.

Given any $\rm{PGL}(3,\mathbb{R})$-positive representation $\rho$, for any non-trivial non-peripheral $\gamma \in \pi_1(S_{g,m})$, the monodromy $\rho(\gamma)$ is conjugate to a totally positive matrix, thus loxodromic. 

For any non-trivial peripheral $\gamma \in \pi_1(S_{g,m})$, the monodromy $\rho(\gamma)$ is conjugate to a totally positive upper triangular matrix. Let $\lambda_1,\lambda_2,\lambda_3$ be the positive diagonal entries where $\lambda_1\geq \lambda_2 \geq \lambda_3>0$.
\end{thm}
Note that the above loxodromic property is also proved by \cite{Lab06} for Anosov representations.

Following the above theorem, we can define $i$-th length for $i=1,2$.
\begin{defn}[$i$-th length]
Given any $\rm{PGL}(3,\mathbb{R})$-positive representation $\rho\in \rm{Pos}_3(S_{g,m})$, for $i=1,2$ and any $\gamma \in \pi_1(S_{g,m})$, we define the {\em $i$-th length} (or called {\em simple root length}) of $\gamma$: 
\[\ell_i(\gamma):=\ell_i^{\rho}(\gamma):=\log \frac{\lambda_i(\rho(\gamma))}{\lambda_{i+1}(\rho(\gamma))}.\]  

Then 
\[\ell(\gamma):=\ell^{\rho}(\gamma):=\ell_1^\rho(\gamma)+\ell_2^\rho(\gamma)\]
is the Hilbert length of $\gamma$ with respect to $\rho$.
\end{defn}

\begin{defn}
Given any $\rm{PGL}(3,\mathbb{R})$-positive representation $\rho \in \rm{Pos}_3(S_{g,m})$, let $\alpha_1,\cdots, \alpha_m$ be the oriented boundary components of the topological surface $S_{g,m}$ such that $S_{g,m}$ is on the left side of $\alpha_s$ for $s=1,\cdots,m$. Let 
\[\mathbf{L}:=\left(\ell_1(\rho(\alpha_1)),\cdots, \ell_1(\rho(\alpha_m)), \ell_2(\rho(\alpha_1)),\cdots, \ell_2(\rho(\alpha_m))\right).\]
We denote the elements in $\rm{Pos}_3(S_{g,m})$ with fixed boundary simple root lengths $\mathbf{L}$ by
$\rm{Pos}_3(S_{g,m})(\mathbf{L})$.

Let us denote $\rm{Pos}_3(S_{g,m})(\mathbf{0})$---the collection of positive representations with unipotent boundary monodromy by $\rm{Pos}_3^u(S_{g,m})$.

Let $\rm{Pos}^h_3(S_{g,m})$ be the collection of positive representations with loxodromic boundary monodromy. Then 
\[\rm{Pos}^h_3(S_{g,m})=\bigcup_{\mathbf{L} \in \mathbb{R}_{>0}^{2m}} \rm{Pos}_3(S_{g,m})(\mathbf{L}).\]

Let $\rm{Pos}_3'(S_{g,m})=\rm{Pos}_3^h(S_{g,m})\cup \rm{Pos}_3^u(S_{g,m})$.
\end{defn}

\begin{defn}[Canonical $\rho$-equivariant map]
\label{definition:cancho}
For any $\rm{PGL}(3,\mathbb{R})$-positive representation $\rho \in \operatorname{Pos}_3^h(S_{g,m})$ with loxodromic boundary monodromy, there is a \emph{canonical $\rho$-equivariant map} $\xi_\rho:\partial_\infty \pi_1(S_{g,m})\rightarrow \mathcal{B}$ such that for any peripheral $\delta \in \pi_1(S_{g,m})$, (by Theorem \ref{theorem:loxo}, $\rho(\delta)$ has eigenvectors $\delta_1,\delta_2,\delta_3$ and the corresponding eigenvalues $\lambda_1,\lambda_2,\lambda_3$ satisfy $\lambda_1>\lambda_2>\lambda_3>0$,) the eigenvectors $(\delta_3,\delta_2,\delta_1)$ ($(\delta_1,\delta_2,\delta_3)$ resp.) form a basis for the flag $\xi_\rho(\delta^-)$ ($\xi_\rho(\delta^+)$ resp.). 

For any $\rho \in \operatorname{Pos}_3^u(S_{g,m})$ with unipotent boundary monodromy, there is only one choice of $\xi_\rho$ (\cite[Theorem 1.14]{FG06}). We also call $\xi_\rho$ the \emph{canonical $\rho$-equivariant map}.
\end{defn}
Any other lift can be obtained by permuting the order of the basis $(\delta_3,\delta_2,\delta_1)$ for the flag $\xi(\delta^-)$ for each $\delta$ as above (see for example \cite[Section 10]{LM09}).

Similar to Theorem \ref{thm:11cor}, for $\rm{Pos}_3^u(S_{g,m})$, we have 
\begin{thm}\cite{Mar10}
\label{theorem:maru}
For the integer $g\geq 2$, the deformation space of (marked) cusped strictly convex $\mathbb{RP}^2$ structures on the surface $S_{g,m}$ is homeomorphic to $\rm{Pos}_3^u(S_{g,m})$, which is a cell of dimension $16g-16+6m$.  
\end{thm}

\begin{remark}
In \cite{Mar12}, Marquis also described the one-to-one correspondence between any $\rho \in \rm{Pos}_3(S_{g,m})\backslash \rm{Pos}_3^u(S_{g,m})$ and the minimal $\rho$-invariant convex $\mathbb{RP}^2$ domain up to equivalence. 
\end{remark}

\begin{remark}
\label{remark:double}
In \cite[Section 9]{LM09}, using the same definition as Definition \ref{definition:Hit}, the notion of the Hitchin representations for closed surfaces are generalized to the representations with loxodromic boundary monodromy. We denote the space of $\operatorname{PGL}(3,\mathbb{R})$-Hitchin representations up to conjugation by $\rm{Hit}_3(S_{g,m})$. By \cite[Theorem 9.1]{LM09}, $\rm{Hit}_3(S_{g,m})\subset \rm{Pos}_3^h(S_{g,m})$. By the gluing process in \cite[Definition 9.2.2.3]{LM09} which satisfies the gluing condition in \cite[Definition 7.2]{FG06}, any $\rho \in \rm{Pos}_3^h(S_{g,m})$ with $2g-2+m\geq 1$ and $m\geq 1$ is glued into a positive representation $d\rho$ for $S_{2g-1+m,0}$ in a canonical way. By \cite[Theorem 1.15]{FG06} $\rm{Hit}_3(S_{2g-1+m,0})= \rm{Pos}_3(S_{2g-1+m,0})$, we have $d\rho$ is a Hitchin representation. Thus $d\rho$ is deformed from a $3$-Fuchsian representation, and the restriction to $S_{g,m}$ induces a deformation path from the $3$-Fuchsian representation for $S_{g,m}$ to $\rho$. Thus $\rho\in\rm{Hit}_3(S_{g,m})$. Hence $\rm{Pos}_3^h(S_{g,m})\subset \rm{Hit}_3(S_{g,m})$. Hence $\rm{Hit}_3(S_{g,m})= \rm{Pos}_3^h(S_{g,m})$.
\end{remark}

\subsection{Projective invariants}
\label{subsection:proinv}
\begin{defn}[triple ratios]
\label{definition:tripleratio}
Consider the triple of flags $(F,G,H)$ in generic position, with bases 
\[
(f_1,f_2,f_3),\;\;(g_1,g_2,g_3),\;\;(h_1,h_2,h_3).
\] 
Then the \emph{triple ratio} $T(F,G,H)$ is defined by:
\begin{align*}
T(F,G,H):= 
\frac
{\Delta\left(f^{2} \wedge g^{1}\right) \Delta\left( g^{2} \wedge h^{1}\right)  \Delta\left(h^{2} \wedge  f^{1}\right)} 
{\Delta\left(f^{2}  \wedge h^{1}\right) \Delta\left(g^{2} \wedge f^{1} \right) \Delta\left(h^{2} \wedge g^{1}\right)}
\end{align*}
where $w^{i}:=w_1\wedge\cdots\wedge w_i$, which is $\rm{PGL}(3,\mathbb{R})$ invariant. 
Notice the symmetry \[T(F,G,H)=T(G,H,F)=T(H,F,G).\]
\end{defn}
Check Figure~\ref{figure:triple} for a geometric description of the triple ratio.

\begin{figure}[ht]
\centering
\includegraphics[scale=0.35]{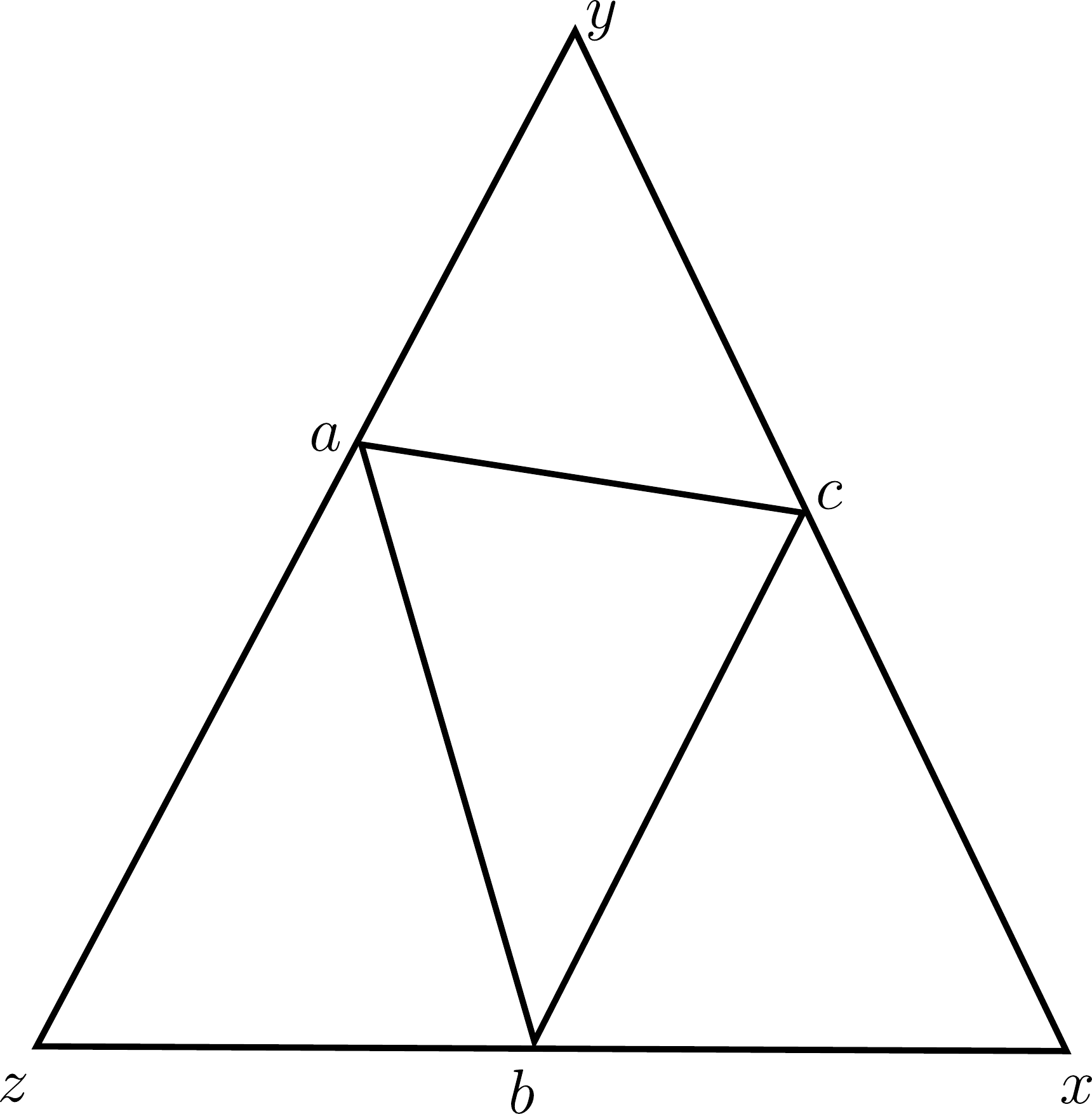}
\small
\caption{The flags are $F=(a,\overline{yz})$, $G=(b,\overline{zx})$, $H=(c,\overline{xy})$. Let $|\cdot|$ be the Euclidean norm. Then the triple ratio $T(F,G,H)=\frac{|ya|}{|az|}\frac{|zb|}{|bx|}\frac{|xc|}{|cy|}$. By Ceva theorem, $T(F,G,H)=1$ if and only if $\overline{ax}$, $\overline{by}$ and $\overline{cz}$ are colinear. }
\label{figure:triple}
\end{figure}

\begin{defn}[Edge functions]
\label{definition:edgefun}
Let $(X,Y,Z,W)$ be the quadruple of flags in generic position, choose their bases 
\[(x_1,x_2,x_3),\;\;(y_1,y_2,y_3),\;\;(z_1,z_2,z_3),\;\;(w_1,w_2,w_3).\]
For $i=1,2$, the \emph{edge functions} are defined to be
\begin{eqnarray*}
&&D_1(X,Y,Z,W):= \frac{\Delta\left(x^{2} \wedge z^{1} \right)}{\Delta\left(x^{2}  \wedge w^{1}\right)}\cdot \frac{ \Delta\left(x^{1} \wedge y^{1} \wedge w^1\right)}{\Delta\left(x^{1}  \wedge y^{1}\wedge z^1\right)}
\end{eqnarray*}
\begin{eqnarray*}
&&D_2(X,Y,Z,W):= \frac{\Delta\left(y^{2} \wedge w^{1} \right)}{\Delta\left(y^{2}  \wedge z^{1}\right)}\cdot \frac{ \Delta\left(x^{1} \wedge y^{1} \wedge z^1\right)}{\Delta\left(x^{1}  \wedge y^{1}\wedge w^1\right)}
\end{eqnarray*}
which are $\rm{PGL}(3,\mathbb{R})$ invariants. 
Notice the symmetry
\[D_1(X,Y,Z,W)=D_2(Y,X,W,Z).\]
\end{defn}

\begin{figure}
\centering
\includegraphics[scale=0.35]{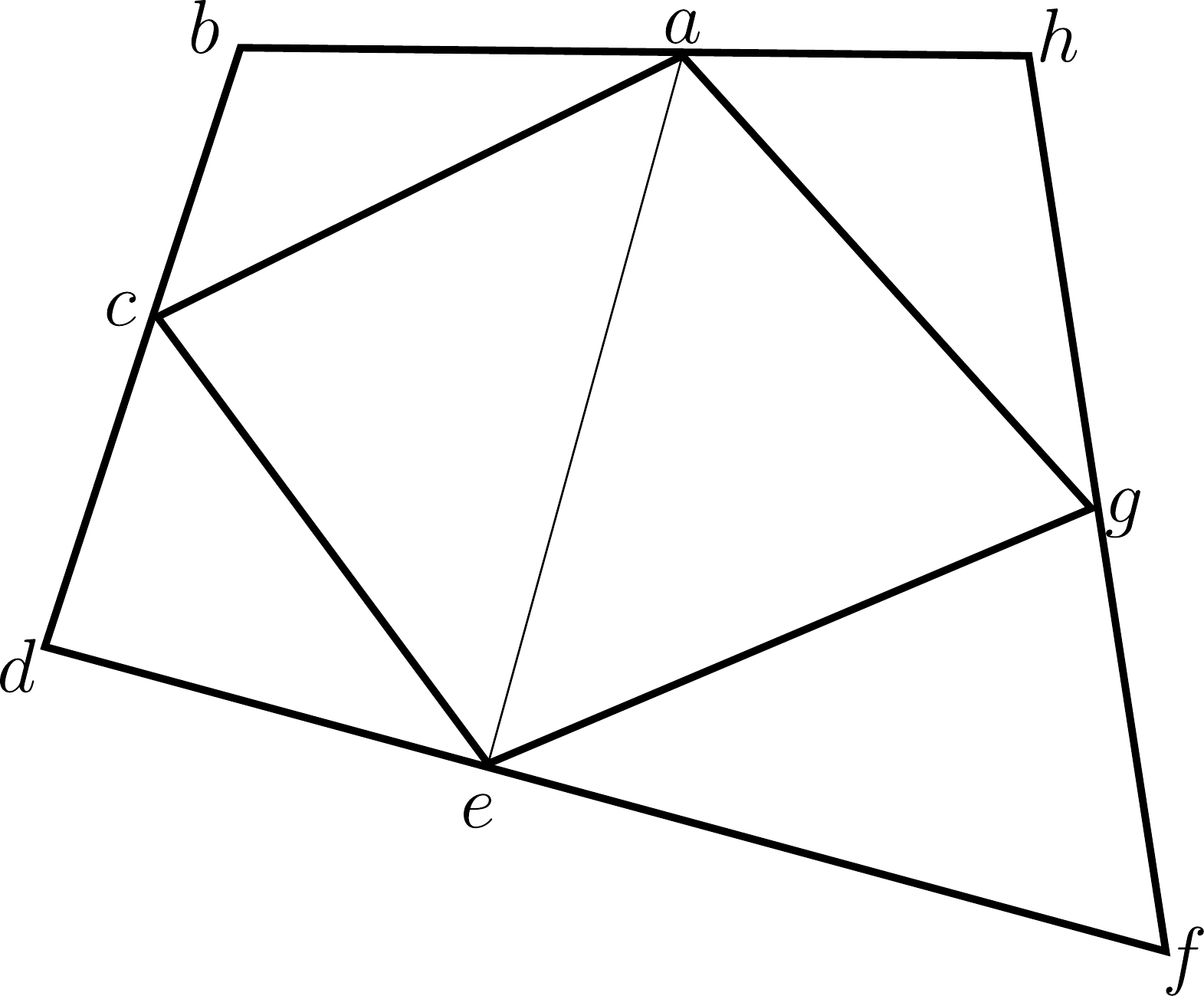}
\small
\caption{The flags are $X=(a,\overline{bh})$, $W=(c,\overline{bd})$, $Y=(e,\overline{df})$, $Z=(g,\overline{fh})$. Up to $\rm{PGL}(3,\mathbb{R})$, the position of $(X,W,Y)$ is decided by the triple ratio $T(X,W,Y)$. The convention for cross ratio in $\mathbb{RP}^1$ is $CR(\alpha,\beta,\gamma,\delta):=\frac{\alpha-\gamma}{\alpha-\delta}\cdot\frac{\beta-\delta}{\beta-\gamma}$. We have 
$D_1(X,Y,Z,W)=CR(\overline{ab},\overline{ae},\overline{ag},\overline{ac})$
deciding the line $\overline{ag}$ and 
$D_2(X,Y,Z,W)=CR(\overline{ef},\overline{ea},\overline{ec},\overline{eg})$
deciding the line $\overline{eg}$, which fix the point $G$. In the end, the line $\overline{hf}$ is decided by the triple ratio $T(X,Y,Z)$. }
\label{figure:edge}
\end{figure}

As shown in Figure~\ref{figure:edge}, the configuration space $\rm{Conf}^+_4$ can be parameterized by the positive numbers 
\[(A,B,C,D):=\left(T(X,W,Y),\;-D_1(X,Y,Z,W),\;-D_2(X,Y,Z,W),\;T(X,Y,Z)\right).\] 
This parametrization depends on the triangulation of the polygon $(a,c,e,g)$. We can choose the triangulation $\{\overline{cg}\}$ instead of $\{\overline{ae}\}$. Then the parameters are changed into
\[(A',B',C',D'):=\left(-D_2(Z,W,Y,X), \;T(Z,X,W),\;T(Z,W,Y),\;-D_1(Z,W,Y,X)\right).\] 
Then by \cite[Section 2]{FG07}, we have 
\[(A',B',C',D')=\left(\frac{1+C}{AC(1+B)},D\frac{1+C+CA+CAB}{1+B+BD+BDC}, A\frac{1+B+BD+BDC}{1+C+CA+CAB}, \frac{1+B}{DB(1+C)} \right).\] 
The space $\rm{Conf}^+_d$ can be understood as a map from a cyclically ordered subset $\mathcal{Q}$ of $S^1$ to $\mathcal{B}$. There is the $d$-gon $D_d$ with $\mathcal{Q}$ as vertices and $S^1$ as the union of edges.
The triangulation above is equivalent to the triangulation of the $d$-gon $D_d$.

\begin{defn}[Parameters]
\label{definition:para}
For the integer $d\geq 3$, let $(x_1,\cdots,x_d)$ be the cyclically ordered set $\mathcal{Q}$. For any anticlockwise ordered triangle $\mathbf{\Delta}:=(x_i,x_j,x_k)$, we define 
\[T(\mathbf{\Delta}):=T(\xi(x_i),\xi(x_j),\xi(x_k)).\] 
For any edge $\mathbf{e}$ with two adjacent anticlockwise ordered triangles $(x_i,x_j,x_k)$ and $(x_i,x_l,x_j)$, we choose an orientation $\overrightarrow{\mathbf{e}}=(x_i,x_j)$, for $i=1,2$, we define
\[D_i(\overrightarrow{\mathbf{e}}):=D_i(\xi(x_i),\xi(x_j),\xi(x_k),\xi(x_l)).\] 

\end{defn}

We can use the above parameters to parameterize $\rm{Conf}^+_d$.
\begin{prop}\cite{FG06}
\label{prop:confpara}
For $d\geq 3$, given a triangulation $\mathcal{T}$ of the $d$-gon $D_d$, let $\Theta$ be the collection of anticlockwise ordered triangles of $\mathcal{T}$. Let $E$ be the collection of edges of $\mathcal{T}$. There exists a real analytic diffeomorphism $\theta:\rm{Conf}^+_d \rightarrow \mathbb{R}_{>0}^{3d-8}$
\[\xi \rightarrow \left(\left(T(\mathbf{\Delta}) \right)_{\Delta\in \Theta},\left(-D_1(\overrightarrow{\mathbf{e}}), -D_2(\overrightarrow{\mathbf{e}}) \right)_{\mathbf{e}\in E}\right).   \]
\end{prop} 
Let us recall the ideal triangulation. For more details, check \cite{Thu79,CB88,PH92,Bon01}.
\begin{defn}{\sc[Ideal triangulation]}
We equip $S_{g,m}$ with a hyperbolic metric $\rho_h
$.
Choose finitely many disjoint simple closed geodesics $\mathcal{P}$ (can be empty). The {\em ideal triangulation $\mathcal{T}$ of $S_{g,m}$ (subordinate to $\mathcal{P}$)} is a simple maximal filling geodesic lamination of $(S_{g,m},\rho_h)$ containing $\mathcal{P}$ with finitely many leaves.
\end{defn}

Let $\mathcal{X}_3(S_{g,m})$ be the space of all the pairs $(\rho,\xi_\rho)$, where $\rho \in \rm{Pos}_3(S_{g,m})$ and $\xi_\rho$ is a $\rho$-equivariant map. 
Considering all the lifts of the ideal triangulation $\mathcal{T}$ into the universal cover, the vertices of all the lifts are contained in $\partial_\infty \pi_1(S_{g,m})$. Using the parameters derived from the images of these vertices under $\xi_\rho$ as Proposition \ref{prop:confpara}, Fock and Goncharov \cite[Theorem 9.1]{FG06} provided a positive atlas for $\mathcal{X}_3(S_{g,m})$. 

\begin{remark}
For $m=0$ when the ideal triangulation $\mathcal{T}$ has non-empty $\mathcal{P}$, Goldman \cite{G90} parameterized the $\rm{PGL}(3,\mathbb{R})$-Hitchin component $\rm{Hit}_3(S_{g,0})$. Then Kim \cite{Kim99} provided a global Darboux coordinate, where some parameters are modified in Choi--Jung--Kim \cite{CJK19}. Using Fock--Goncharov's parameters \cite{FG06} in Definition \ref{definition:para}, Bonahon and Dreyer \cite{BD14} parameterized $\rm{Hit}_3(S_{g,0})$ with respect to an ideal triangulation $\mathcal{T}$ on the closed surface $S_{g,0}$ and a choice of transverse arcs to $\mathcal{P}$. Based on the Bonahon--Dreyer's parametrization, Sun--Wienhard--Zhang \cite{SWZ17,SZ17} provided a global Darboux coordinate with respect to an ideal triangulation $\mathcal{T}$ subordinate to a pants decomposition $\mathcal{P}$ and a choice of transverse arcs to $\mathcal{P}$. Later on, we will use the last mentioned global Darboux coordinate system for our computation. 
\end{remark}

\section{Bounded moduli spaces}
In this section, we introduce many subspaces of the moduli space of unmarked convex $\mathbb{RP}^2$ structures on $S_{g,m}$ with some natural boundedness conditions. Many of them are inspired by the work of Benoist \cite{Ben03} and Colbois--Vernicos--Verovic \cite{CVV08}. Each one of them is not compact, because it contains the entire $3$-Fuchsian locus that is isomorphic to the moduli space of Riemann surfaces. We mainly interested in the area bounded subset and the projective invariants bounded subset.
\subsection{Area boundedness and projective invariants boundedness}
\label{subsection:abandb}
Given any $\rm{PGL}(3,\mathbb{R})$-positive representation $\rho$ with loxodromic boundary monodromy for $S_{g,m}$ with $m\geq 1$, let $\Omega\subset \mathbb{RP}^2$ be the minimal $\rho(\pi_1(S_{g,m}))$-invariant convex domain in $\mathbb{RP}^2$ as in Figure \ref{figure:cdd}. 
\begin{figure}
\centering
\includegraphics[scale=0.35]{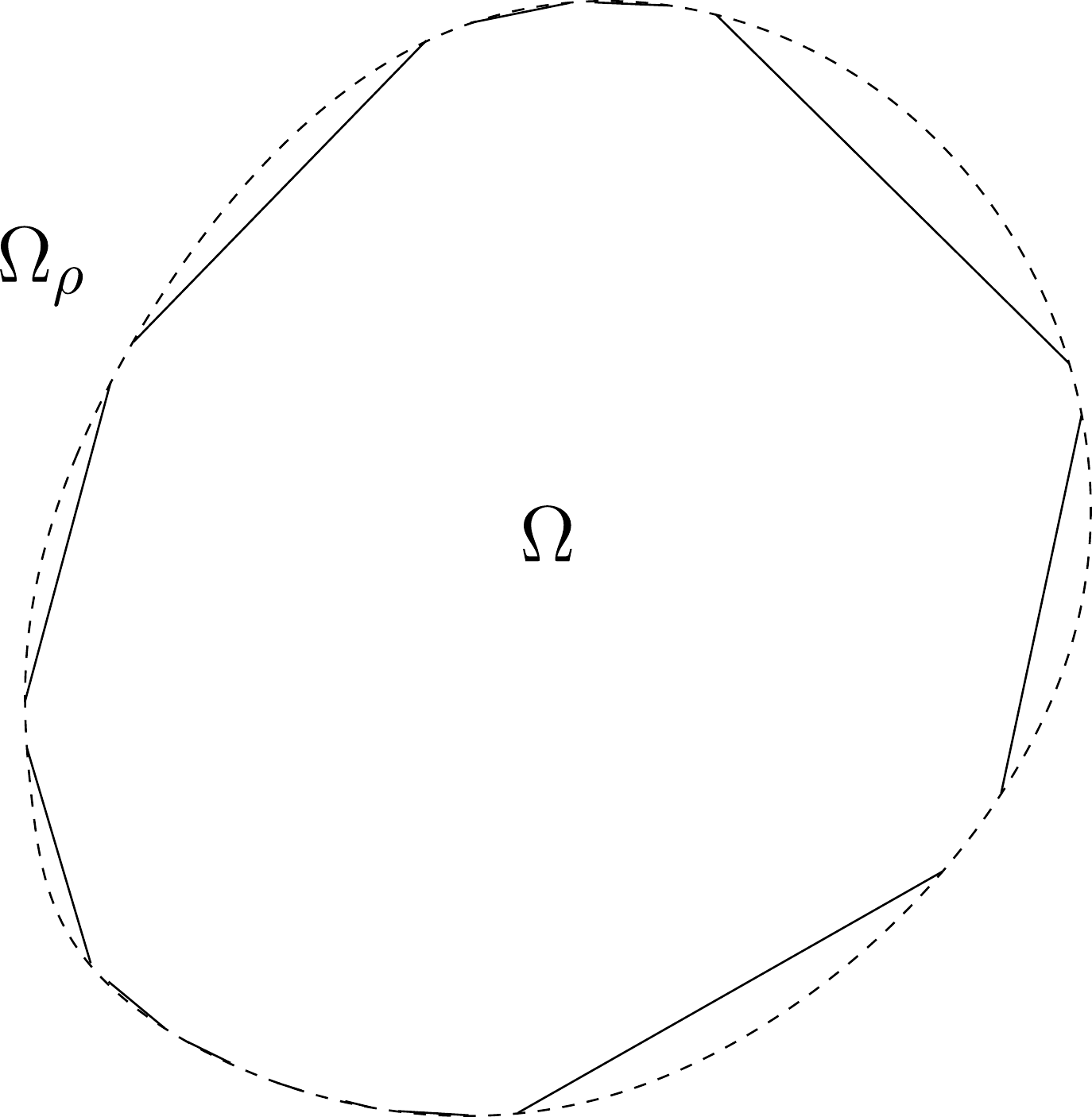}
\small
\caption{The $d\rho(\pi_1(S_{g,m}))$-invariant convex domain $\Omega_\rho$ is bounded by the dotted curve. The minimal $\rho(\pi_1(S_{g,m}))$-invariant convex domain $\Omega$ is obtained by removing infinite many hemispheres illustrated by the straight line segments from $\Omega_\rho$.}
\label{figure:cdd}
\end{figure}
By \cite{Mar12}, the area of $S_{g,m}$ with respect to the Hilbert metric on $\Omega$ is infinite. But $\Omega$ is not the natural one to use. Indeed, when $\rho$ is $3$-Fuchsian, the minimal $\rho(\pi_1(S_{g,m}))$-invariant convex domain $\Omega$ is the universal cover $ \widetilde{S_{g,m}}$ with respect to $\rho$ (considered as a hyperbolic metric), then the Hilbert metric on $\Omega$ is not the hyperbolic metric on the surface with geodesic boundary with respect to the Klein model. For any $\rho \in \rm{Pos}_3^h(S_{g,m})$, let us consider the unique representation $d\rho$ introduced in \cite[Definition 9.2.2.3]{LM09} by doubling the surface $S_{g,m}$, then the $d\rho(\pi_1(S_{g,m}))$-invariant convex domain, denoted by $\Omega_\rho$, is unique. When $\rho$ is $3$-Fuchsian, the convex domain $\Omega_\rho$ a disk up to projective transformations, thus the Hilbert metric on $\Omega_\rho$ is indeed the hyperbolic metric on the surface $S_{g,m}$ with respect to the Klein model. As a subsurface of the doubled surface, the area of $S_{g,m}$ with respect to $\Omega_\rho$ is finite.

\begin{defn}[Canonical area]
\label{definition:carea}
For any $\rho \in \rm{Pos}_3^h(S_{g,m})$ with loxodromic boundary monodromy, let us consider its double $d \rho \in \rm{Pos}_3(S_{2g-1+m,0})$ by \cite[Definition 9.2.2.3]{LM09}, then we define the {\em canonical convex domain} $\Omega_{\rho}$ to be the unique $d\rho$-invariant strictly convex domain.
 
For any $\rho \in \rm{Pos}_3^u(S_{g,m})$ with unipotent boundary monodromy, we define the {\em canonical convex domain} $\Omega_{\rho}$ to be the unique $\rho(\pi_1(S_{g,m}))$-invariant strictly convex domain.  
 
For $\rho\in \rm{Pos}_3^h(S_{g,m})\cup \rm{Pos}_3^u(S_{g,m})=\rm{Pos}_3'(S_{g,m})$, we define the {\em canonical area} of $(S_{g,m},\rho)$ to be the area of $S_{g,m}$ with respect to the canonical convex domain $\Omega_\rho$.
\end{defn}


For all the ideal triangulation, the Thurston's shearing coordinates for the Teichm\"uller space are not bounded within an interval. For $\rho\in \rm{Pos}_3'(S_{g,m})$, we will show the uniformly boundedness of some other projective invariants for a subset of $\rm{Pos}_3'(S_{g,m})$ where the canonical areas are uniformly bounded. 

\begin{defn}
For any $\rho\in \rm{Pos}_3'(S_{g,m})$, let us consider the images of the canonical $\rho$-equivariant map $\xi_\rho$ in order to define the triple ratios and edge functions as in Definitions \ref{definition:tripleratio}, \ref{definition:edgefun}. Given any ideal triangulation $\mathcal{T}$, for any lift $\mathbf{\widetilde{\Delta}}$ of the ideal triangle $\mathbf{\Delta}$, let 
\[T(\mathbf{\Delta})(\rho):=T(\mathbf{\Delta}):= T(\mathbf{\widetilde{\Delta}}),\] 
For any lift $\widetilde{\overrightarrow{\mathbf{e}}}$ of the ideal edge $\overrightarrow{\mathbf{e}}$, let
\[D_i(\overrightarrow{\mathbf{e}})(\rho):=D_i(\overrightarrow{\mathbf{e}}):= D_i(\widetilde{\overrightarrow{\mathbf{e}}}).\] 
\end{defn} 
\textbf{Triangle invariant boundedness}

The following is basically a consequence of \cite[Proposition 0.3]{AC18}.
\begin{prop}
\label{proposition:trib}
For any $\rho\in \rm{Pos}_3'(S_{g,m})$ such that the canonical area of $(S_{g,m},\rho)$ is bounded above by a positive constant $t$, for any anticlockwise ordered ideal triangle $\mathbf{\Delta}$ in any ideal triangulation $\mathcal{T}$ of $S_{g,m}$, there exists a constant $T(t)$ which is a polynomial of $t$ and does not depend on $\rho$ such that 
\[\left|\log T(\mathbf{\Delta})(\rho)\right|\leq T(t). \]
 
\end{prop}

\begin{proof}
By \cite[Proposition 0.3]{AC18}, we have the $p$-area
\[parea_{\Omega_{\rho}}(\mathbf{\Delta})\geq \frac{\pi^2+(\log T(\mathbf{\Delta})(\rho))^2}{8}. \]
By \cite{B60}, the area is comparable to $p$-area. Thus there is a universal constant $C$ which does not depend on $\rho$ and $\mathbf{\Delta}$ such that
\[parea_{\Omega_{\rho}}(\mathbf{\Delta})\leq C\cdot  area_{\Omega_{\rho}}(\mathbf{\Delta}).\]
Hence
\[\left|\log T(\mathbf{\Delta})(\rho)\right|\leq \sqrt{8C \cdot area_{\Omega_{\rho}}(\mathbf{\Delta})-\pi^2}\leq \sqrt{8C t-\pi^2}\leq 8C t-\pi^2+1\]
for any anticlockwise ordered ideal triangle $\mathbf{\Delta}$ in any ideal triangulation $\mathcal{T}$. We take $T(t):=8C t-\pi^2+1$.
\end{proof}
\begin{remark}
For any $n$ in general and a given $\rho \in  \rm{Pos}_n(S_{g,m})$, the boundedness of triple ratios for any ideal triangle in any ideal triangulation is proved in \cite[Theorem 3.4]{HS19}.
\end{remark}
\textbf{Bulging invariant boundedness}

Bulging deformation was introduced by Goldman in \cite{G13}. It corresponds to deform linearly the difference of the log of two edge functions along one edge. The following is essentially a consequence of \cite[Proposition 4.2]{Ki18}.
\begin{prop}
\label{proposition:bulb}
For any $\rho\in\rm{Pos}_3^h(S_{g,m})$ such that the canonical area of $(S_{g,m},\rho)$ is bounded above by a positive constant $t$, for any ideal quadrilateral embedded in a pair of pants and its oriented diagonal ideal edge $\overrightarrow{\mathbf{e}}$ in any ideal triangulation $\mathcal{T}$ of $S_{g,m}$, there exists a constant $D(t)$ which does not depend on $\rho$ such that 
\[\left|\log D_1(\overrightarrow{\mathbf{e}})(\rho)-\log D_2(\overrightarrow{\mathbf{e}})(\rho)\right|\leq D(t). \]
\end{prop}
\begin{proof}
In \cite[Proposition 4.2]{Ki18}, Kim proved that no matter how one deforms the representation $\rho$, given an ideal quadrilateral $(x,w,y,z)$ and its oriented diagonal ideal edge $\overrightarrow{\mathbf{e}}=(x,y)$, if $\left|\log D_1(\overrightarrow{\mathbf{e}})(\rho)-\log D_2(\overrightarrow{\mathbf{e}})(\rho)\right|$ goes to infinity, then the canonical area of the ideal quadrilateral converges to infinity. (The argument goes as follows. We denote the tangent line at $u\in \partial \Omega$ by $u^*$. Suppose $z$ lies in the triangle $(x,y,x^*\cap y^*)$. By the freedom of the projective transformations, we fix one side of $(x,y)$ that contains $w$. When the bulging invariant goes to infinity, $z$ converges to $x^*\cap y^*$. Thus the canonical area of the ideal quadrilateral goes to infinity.) Hence when the canonical area of the ideal quadrilateral is bounded above by $t$, $\left|\log D_1(\overrightarrow{\mathbf{e}})(\rho)-\log D_2(\overrightarrow{\mathbf{e}})(\rho)\right|$ is bounded above by a constant $d(t)$.

For any element $\delta$ in mapping class group, we have
\[\left|\log D_1(\delta\overrightarrow{\mathbf{e}})(\rho)-\log D_2(\delta\overrightarrow{\mathbf{e}})(\rho)\right|=\left|\log D_1(\overrightarrow{\mathbf{e}})(\delta\rho)-\log D_2(\overrightarrow{\mathbf{e}})(\delta\rho)\right|.\]
Since the canonical area of $(S_{g,m},\delta\rho)$ is the same as the canonical area of $(S_{g,m},\rho)$, we have 
\[\left|\log D_1(\delta\overrightarrow{\mathbf{e}})(\rho)-\log D_2(\delta\overrightarrow{\mathbf{e}})(\rho)\right|=\left|\log D_1(\overrightarrow{\mathbf{e}})(\delta\rho)-\log D_2(\overrightarrow{\mathbf{e}})(\delta\rho)\right|\]
 is uniformly bounded above by $d(t)$ for any $\delta$.

Moreover, there are only finitely many mapping class group orbits of the embedded pairs of pants. Thus there are finitely many ideal quadrilaterals embedded in a pair of pants and their oriented diagonal ideal edges $\overrightarrow{\mathbf{e}}=\overrightarrow{\mathbf{e}}_1,\cdots,\overrightarrow{\mathbf{e}}_k$ up to the mapping class group actions. Suppose that $d_i(t)$ for $\overrightarrow{\mathbf{e}}_i$ is defined similarly as $d(t)$ for $\overrightarrow{\mathbf{e}}$. Let $D(t):=\max\{d_1(t),\cdots,d_k(t)\}$. Then for any $\rho\in \rm{Pos}_3(S_{g,m})$ such that the canonical area of $(S_{g,m},\rho)$ is bounded above by a positive constant $t$, for any ideal quadrilateral embedded in a pair of pants and its oriented diagonal ideal edge $\overrightarrow{\mathbf{e}}$ in any ideal triangulation $\mathcal{T}$ of $S_{g,m}$, the number $\left|\log D_1(\overrightarrow{\mathbf{e}})(\rho)-\log D_2(\overrightarrow{\mathbf{e}})(\rho)\right|$
is bounded above by a constant $D(t)$.
\end{proof}

\begin{conjecture}
\label{conj:D}
The constant $D(t)$ above can be a polynomial of $t$.
\end{conjecture}
We suggest to investigate the canonical area of a quadrilateral to obtain the above conjecture.

Before we continue the uniformly boundedness of some other projective invariants, let us recall the definitions that are used to describe the shape of convex domain $\Omega$. 
\begin{defn}\cite[$\alpha$-H\"older and $\beta$-convex]{Ben04} 
Let $\Omega\subset\mathbb{R}^2\subset\mathbb{RP}^2$ be a convex open subset of $\mathbb{RP}^2$ and fix an Euclidean metric $d_E$ on $\mathbb{R}^2$. We say that $\partial\Omega$ is \emph{$\alpha$-H\"older}, for $\alpha\in(1,2]$, if for every compact subset $K\subset\partial\Omega$, there exists a constant $C_K>0$ such that, for all $p,q\in K$, we have:
\begin{align*}
d_E(q, T_p\partial\Omega)\leq C_K\cdot d_E(q,p)^\alpha.
\end{align*}
We say that $\partial\Omega$ is \emph{$\beta$-convex}, for $\beta\in[2,+\infty)$, if there exists a constant $C>0$ such that for all $p,q\in\partial\Omega$, we have:
\begin{align*}
d_E(q, T_p\partial\Omega)\geq C \cdot d_E(q,p)^\beta.
\end{align*}
\end{defn}

\begin{defn}\cite[quasisymmetric]{Ben03}
We say that a $C^1$ convex function $f:I\rightarrow J$ between two intervals of $\mathbb{R}$ is {\em $H$-quasisymmetrically convex} if for any $x-h,x+h \in I$, we have
\[|f(x+h)-f(x)-f'(x)h | \leq H |f(x-h)-f(x)+f'(x)h|.\]

We say that a continuous function $f:I\rightarrow J$ between two intervals of $\mathbb{R}$ is {\em $H$-quasisymmetric} if for any $x-h,x+h \in I$, we have
\[|f(x+h)-f(x) | \leq H |f(x-h)-f(x)|.\]

Let $F$ be the graph function of the $C^1$ convex curve $\partial\Omega$, we say that $\Omega$ or $F$ is 
\begin{enumerate}
\item {\em $H$-quasisymmetrically convex} if the function $F$ is $H$-quasisymmetrically convex on any compact interval.
\item {\em derivative $H$-quasisymmetrically convex} if the function $F'$ is $H$-quasisymmetric on any compact interval.
\end{enumerate}
\end{defn}
By \cite[Proposition 5.2]{Ben03}, a convex fonction $f$ is quasisymmetrically convex on any compact interval if and only if its derivative $f'$ is quasisymmetric on any compact interval.

\textbf{$\ell_2/\ell_1$ invariant boundedness}

\begin{prop}
\label{proposition:l1l2b}
For any given $\rho\in \rm{Pos}_3^h(S_{g,m})$ such that the canonical area of $(S_{g,m},\rho)$ is bounded above by a positive constant $t$, for any non-trivial non-peripheral $\gamma\in \pi_1(S_{g,m})$, there exists a constant $L(t)$ which does not depend on $\rho$ such that
\[\frac{\ell_2^\rho(\gamma)}{\ell_1^\rho(\gamma)}\leq L(t).\]
\end{prop}
\begin{proof}
If $\rho \in \rm{Pos}_3^h(S_{g,m})$, after doubling, we have the boundary of convex domain $\Omega_{\rho}$ is both $\alpha$-H\"older and $\beta$-convex.
Then \cite[Corollary 5.3]{Ben04} states that for any $\gamma \in \pi_1(S_{g,m})$ that is non-trivial and non-peripheral, we have
\[\frac{\ell_2^\rho(\gamma)}{\ell_1^\rho(\gamma)}\leq \min\left\{\beta-1,\frac{1}{\alpha-1}\right\}.\]

Now for any $\rho\in \rm{Pos}_3^h(S_{g,m})$ such that the area of $(S_{g,m},\rho)$ is bounded above by a positive constant $t$, the area of any ideal triangle is also bounded above by $t$. By \cite[Theorem 2]{CVV08}, there exists a constant $C>0$ such that $\Omega_{\rho}$ is $Ct$-(Gromov) hyperbolic for any such $\rho$. By \cite[Proposition 5.2, 6.6]{Ben03}, any $\Omega_\rho$ that is $Ct$-hyperbolic implies that there exists a $H(Ct)>0$ such that $\Omega_\rho$ is derivative $H(Ct)$-quasisymmetrically convex. Following from \cite[Lemma 4.9]{Ben03}, there exists $\alpha(Ct)\in(1,2]$ and $\beta(Ct)\in[2,+\infty)$ such that $\partial \Omega_\rho$ is both $\alpha(Ct)$-H\"older and $\beta(Ct)$-convex for any such $\rho$. Let $L(t)=\min\left\{\beta(Ct)-1,\frac{1}{\alpha(Ct)-1}\right\}>0$. We conclude that, for any $\rho\in  \rm{Pos}_3^h(S_{g,m})$ such that the canonical area of $(S_{g,m},\rho)$ is bounded above by a positive constant $t$, for any non-trivial and non-peripheral $\gamma\in \pi_1(S_{g,m})$, there exists a constant $L(t)$ such that
\[\frac{\ell_2^\rho(\gamma)}{\ell_1^\rho(\gamma)}\leq L(t).\]
\end{proof}

\begin{conjecture}
\label{conj:L}
The constant $L(t)$ above can be a polynomial of $t$.
\end{conjecture}
We suggest to investigate the quantitative relation between $\delta$-hyperbolic and derivative $H$-quasisymmetrically convex in \cite[Proposition 6.6]{Ben03}.

\subsection{Twist flows} 
\label{subsection:twist}
In \cite{G86}, Goldman introduced the {\em twist flow} $\{\phi_s\}_{s\in\mathbb{R}}$ along the non-peripheral oriented simple closed geodesic $\gamma$ on the representation space. For any positive $\operatorname{PGL}(3,\mathbb{R})$ representation $\rho$, let $(v_1,v_2,v_3)$ be the eigenvectors for the eigenvalues $\lambda_1,\lambda_2,\lambda_3$ of $\rho(\gamma)$ where $\lambda_1>\lambda_2>\lambda_3>0$. For $a+b+c=0$, let $g_t$ be the projective transformation
\[\left(\begin{array}{ccc}
e^{at} & 0 & 0 \\ 
0 & e^{bt} & 0 \\ 
0 & 0 & e^{ct}
\end{array} \right)\]
with respect to the basis $(v_1,v_2,v_3)$.
\begin{itemize}
\item When the oriented simple closed geodesic $\gamma\in \pi_1(S)$ is separating, $\gamma$ cuts $S$ into two connected surfaces $S_1$ and $S_2$. The group $\pi_1(S)=\pi_1(S) *_{<\gamma>} \pi_2(S)$ is amalgamated over the cyclic group $<\gamma>$. We define
\begin{equation*}
\phi_s \rho(\delta)= \begin{cases}

   \rho(\delta) &\mbox{if $\delta\in \pi_1(S_1)$,}\\

   g_t \rho(\delta) g_t^{-1} &\mbox{if $\delta\in \pi_1(S_2)$.}
\end{cases} 
\end{equation*}
\item When the oriented simple closed geodesic $\gamma\in \pi_1(S)$ is non-separating, $\gamma$ cuts $S$ into $S\backslash \gamma$ with two extra boundary components $\gamma_+$ and $\gamma_-$. The group $\pi_1(S)$ is generated by the subgroup $\pi_1(S\backslash \gamma)$ and $\beta$ with the relation $\beta\gamma_+\beta\gamma_-=1$. We define
\begin{equation*}
\phi_s \rho(\delta)= \begin{cases}

   \rho(\delta) &\mbox{if $\delta\in \pi_1(S\backslash \gamma)$,}\\

   \rho(\beta) g_t  &\mbox{$\delta=\beta$.}
\end{cases} 
\end{equation*}
\end{itemize}
By \cite[Theorem 4.5, 4.7]{G86}, Goldman showed that the above flow defined on the representation space $\rm{Hom}(\pi_1(S),\operatorname{PGL}(3,\mathbb{R}))$ descends to a Hamiltonian flow on the representation variety $\rm{Hom}(\pi_1(S),\operatorname{PGL}(3,\mathbb{R}))/\operatorname{PGL}(3,\mathbb{R})$.
Any two different twist flows along $\gamma$ are commuting with each other. We denote 
\begin{enumerate}
\item the twist flow along $\gamma$ for $(a,b,c)=(\frac{2}{3},-\frac{1}{3},-\frac{1}{3} )$ by $\theta_1$, and 
\item the twist flow along $\gamma$ for $(a,b,c)=(\frac{1}{3},\frac{1}{3},-\frac{2}{3} )$ by $\theta_2$.
\end{enumerate} 
Then $\theta_1$, $\theta_2$, $\frac{1}{2}(\theta_1+\theta_2)$ are the Hamiltonian flows for $\ell_1$, $\ell_2$, $\ell$ length of $\gamma$ respectively. The twist flow $\theta_2-\theta_1$ is called the {\em twist-bulging flow}.
\begin{prop}%
\label{proposition:tbu}
For any given $\rho\in \rm{Pos}_3^h(S_{g,m})$ such that the canonical area of $(S_{g,m},\rho)$ is bounded above by a positive constant $t$, for any non-trivial non-peripheral simple $\gamma\in \pi_1(S_{g,m})$, suppose $b^\rho(\gamma)$ is the maximal of the absolute value of the twist-bulging deforming parameter of $\rho$ along $\gamma$ such that the resulting representation still lies in $\rm{Pos}_3^h(S_{g,m})$ with the canonical area bounded above $t$. There exists a polynomial $B(t)$ which does not depend on $\rho$ and $\gamma$ such that
\[b^\rho(\gamma)
\leq  B(t).\]
\end{prop}
\begin{proof}
For any non-trivial non-peripheral simple $\gamma\in \pi_1(S_{g,m})$, let us consider the representation $\rho_1$ obtained by twist-bulging to the maximal from $\rho$.
By the proof of \cite[Theorem 3.7]{FK16}, there is a special cylinder neighbourhood around $\gamma$ with injective radius $\ell_\gamma$ and length $C b^\rho(\gamma)$ for some constant $C>0$. To estimate the area of the cylinder, we can fill the cylinder by $[C b^\rho(\gamma)/(2(\ell_\gamma+1))]$ discs. By \cite[Proposition 3.2]{BH13}, the Riemannian Blaschke metric is uniformly comparable to the Hilbert metric. Thus the area of the disc is approximately $e^{c\ell_\gamma}$ where $c$ is a constant. Then the area of the cylinder is bounded below by $C_0 [C b^\rho(\gamma)/(2(\ell_\gamma+1))] e^{c\ell_\gamma}$ and bounded above by $t$. Thus there is a polynomial $B(t)$ which does not depend on $\rho$ and $\gamma$ such that $b^\rho(\gamma)\leq B(t)$. 
\end{proof}

\subsection{Bounded moduli spaces}
\label{section:boundedm}
Propositions \ref{proposition:trib}, \ref{proposition:bulb}, \ref{proposition:l1l2b} and \ref{proposition:tbu} suggest us to define the following mapping class group invariant subsets of $\rm{Pos}_3(S_{g,m})(\mathbf{L})\subset \rm{Pos}^h_3(S_{g,m})$.
\begin{defn}[Bounded subsets]
\label{definition:bpos}
Given $\rho \in \rm{Pos}'_3(S_{g,m})$,
\begin{enumerate}
\item let $mT(\rho)$ be the maximal value of $\left|\log T(\mathbf{\Delta})(\rho)\right|$ for any anticlockwise ordered ideal triangle $\mathbf{\Delta}$ in any ideal triangulation $\mathcal{T}$ of $S_{g,m}$;
\item let $mD(\rho)$ be the maximal value of $\left|\log D_1(\overrightarrow{\mathbf{e}})(\rho)-\log D_2(\overrightarrow{\mathbf{e}})(\rho)\right|$ for any ideal quadrilateral embedded in one pair of pants and its oriented diagonal ideal edge $\overrightarrow{\mathbf{e}}$ in any ideal triangulation $\mathcal{T}$ of $S_{g,m}$;
\item let $mL(\rho)$ be the maximal value of $\frac{\ell_2^\rho(\gamma)}{\ell_1^\rho(\gamma)}$ for any non-trivial non-peripheral $\gamma\in \pi_1(S_{g,m})$.
\item let $mB(\rho)$ be the maximal value of $ b^\rho(\gamma)$ for non-trivial non-peripheral $\gamma\in \pi_1(S_{g,m})$.

\end{enumerate}
The {\em $t$-bounded subset} $\rm{Pos}_3^{t}(S_{g,m})(\mathbf{L})$ of $\rm{Pos}_3(S_{g,m})(\mathbf{L})$ is the collection of these $\rho\in \rm{Pos}_3(S_{g,m})(\mathbf{L})$ such that $mT(\rho)$, $mD(\rho)$, $mL(\rho)$ and $mb(\rho)$ are bounded above by $t$.\\
The {\em $t$-area bounded subset} $\rm{APos}_3^{t}(S_{g,m})(\mathbf{L})$ of $\rm{Pos}_3(S_{g,m})(\mathbf{L})$ is the collection of these $\rho\in \rm{Pos}_3(S_{g,m})(\mathbf{L})$ such that the canonical area of $(S_{g,m},\rho)$ is bounded above by $t$. We have the mapping class group invariant exhaustions
\[\rm{Pos}_3(S_{g,m})(\mathbf{L})=\bigcup_{t>0} \rm{Pos}_3^{t}(S_{g,m})(\mathbf{L}),\;\;\;\rm{Pos}_3(S_{g,m})(\mathbf{L})=\bigcup_{t>0} \rm{APos}_3^{t}(S_{g,m})(\mathbf{L}).\]
\end{defn}
Proved by Goldman \cite{G90} and Labourie \cite{Lab08}, the mapping class group $\rm{Mod}(S_{g,m})$ acts on $\rm{Pos}_3(S_{g,m})(\mathbf{L})$ properly and discontinuously. Thus the quotient is well defined. 
We are ready to introduce our main objects that we study.
\begin{defn}
\label{definition:mainob}
The {\em moduli space of unmarked positive convex $\mathbb{RP}^2$ structures} on $S_{g,m}$ with boundary simple root lengths $\mathbf{L}$ is $\rm{Pos}_3(S_{g,m})(\mathbf{L})/\rm{Mod}(S_{g,m})$, denoted by $\mathcal{H}(S_{g,m})(\mathbf{L})$.

The {\em moduli space of unmarked $t$-bounded positive convex $\mathbb{RP}^2$ structures} on $S_{g,m}$ with boundary simple root lengths $\mathbf{L}$ is $\rm{Pos}_3^t(S_{g,m})(\mathbf{L})/\rm{Mod}(S_{g,m})$, denoted by $\mathcal{H}^t(S_{g,m})(\mathbf{L})$.

The {\em moduli space of unmarked $t$-area bounded positive convex $\mathbb{RP}^2$ structures} on $S_{g,m}$ with boundary simple root lengths $\mathbf{L}$ is $\rm{APos}_3^{t}(S_{g,m})(\mathbf{L})/\rm{Mod}(S_{g,m})$, denoted by $\mathcal{AH}^{t}(S_{g,m})(\mathbf{L})$.
\end{defn}
Let 
\[\mathcal{H}(S_{g,m})(\mathbf{L})=\bigcup_{t>0} U^{t}(S_{g,m})(\mathbf{L}),\;\;\;\mathcal{H}(S_{g,m})(\mathbf{L})=\bigcup_{t>0} W^{t}(S_{g,m})(\mathbf{L})\]
be two exhaustions of $\mathcal{H}(S_{g,m})(\mathbf{L})$. We want to compare two exhaustions in the following way.
\begin{defn}[Comparable]
We say that the subset $U^{t}(S_{g,m})(\mathbf{L})$ is {\em comparable} to $W^{t}(S_{g,m})(\mathbf{L})$ if there exist $c(t)$ and $C(t)$ such that
\[U^{c(t)}(S_{g,m})(\mathbf{L})\subset W^{t}(S_{g,m})(\mathbf{L})\subset U^{C(t)}(S_{g,m})(\mathbf{L}).\] 
Moreover, if both $c(t)$ and $C(t)$ are polynomial (exponential resp.) function of $t$, we say that $U^{t}(S_{g,m})(\mathbf{L})$ is {\em polynomially (exponentially resp.) comparable} to $W^{t}(S_{g,m})(\mathbf{L})$.
\end{defn}
Propositions \ref{proposition:trib}, \ref{proposition:bulb}, \ref{proposition:l1l2b} and \ref{proposition:tbu} implies that
\begin{cor}
\label{corollary:AHH}
$\mathcal{AH}^{t}(S_{g,m})(\mathbf{L})\subset \mathcal{H}^{C(t)}(S_{g,m})(\mathbf{L})$. 
\end{cor}
Moreover, if Conjectures \ref{conj:D} and \ref{conj:L} are true, the above function $C(t)$ is polynomial of $t$.

After the work of Benoist \cite{Ben03}, there are many other subsets of $\mathcal{H}(S_{g,m})(\mathbf{L})$ that provide exhaustions. Let us recall the following projective invariant first. 
\begin{defn}\cite[Definition 5.11]{Ben03}
The {\em harmonic quadruplet} is a cyclically ordered quadruplet $(a,b,c,d)\in \partial \Omega_\rho$ such that $\overline{ac}$, the tangent line $b^*$ at $b$ and the tangent line $d^*$ at $d$ cross the same point, denoted by $y$. Let $x=\overline{ac}\cap \overline{bd}$. The {\em cross ratio} of the harmonic quadruplet is 
\[\psi(a,b,c,d):=\frac{|xc|}{|ax|}\cdot \frac{|ay|}{|cy|}.\]
\end{defn}

\begin{remark}
\label{remark:triharm}
The point $a$ is determined by the line crossing $y=b^*\cap d^*$ and $c$. Thus any ordered triple $(b,c,d)$ determines the harmonic quadruplet $(a,b,c,d)$. Hence, like the triple ratio, the function $\psi(a,b,c,d)$ is also a projective invariant of ordered triple of points. The two functions are closely related to each other.
\end{remark}

\begin{example}[Other subsets providing exhaustions]
\label{example:exhaustion}
Let us consider the collection of these $\rho\in \mathcal{H}(S_{g,m})(\mathbf{L})$ such that:
\begin{enumerate}
\item the canonical area of any ideal triangle with respect to $\rho$ is bounded above by $t$, denoted by $A^t(S_{g,m})(\mathbf{L})$;
\item the canonical convex domain $\Omega_\rho$ is $t$-hyperbolic, denoted by $B^t(S_{g,m})(\mathbf{L})$;
\item the canonical convex domain $\Omega_{\rho}$ is derivative $t$-quasisymmetrically convex, denoted by $C^t(S_{g,m})(\mathbf{L})$;
\item the boundary $\partial \Omega_{\rho}$ of the canonical convex domain is $t$-H\"older, denoted by $D^t(S_{g,m})(\mathbf{L})$;
\item the boundary $\partial \Omega_{\rho}$ of the canonical convex domain is $t$-convex, denoted by $E^t(S_{g,m})(\mathbf{L})$;
\item the maximal of the logs of the cross ratios of all the harmonic quadruplets is bounded above by $t$, denoted by $F^t(S_{g,m})(\mathbf{L})$;
\item the function $mT(\rho)$ in Definition \ref{definition:bpos} is bounded above by $t$, denoted by $G^t(S_{g,m})(\mathbf{L})$.
\end{enumerate}
\end{example}

\begin{remark}
\label{remark:comparable}
Some qualitative results among these subsets are known, but very few quantitative results are known.
\begin{enumerate}
\item 
Obviously, the subset $A^t(S_{g,m})(\mathbf{L})$ is polynomially comparable to $\mathcal{AH}^{t}(S_{g,m})(\mathbf{L})$.
\item By definition, we have $\mathcal{H}^t(S_{g,m})(\mathbf{L})\subset G^t(S_{g,m})(\mathbf{L})$.
\item By Proposition \ref{proposition:trib}, we have $A^t(S_{g,m})(\mathbf{L})\subset G^{T(t)}(S_{g,m})(\mathbf{L})$ where $T(t)$ is a polynomial of $t$.
\item The subset $A^t(S_{g,m})(\mathbf{L})$ is comparable to $B^t(S_{g,m})(\mathbf{L})$ by \cite[Theorem 1]{CVV08}. 
\item By \cite[Proposition 3.2]{Ben03}, the subset $B^t(S_{g,m})(\mathbf{L})$ is comparable to $F^t(S_{g,m})(\mathbf{L})$. 
\item By \cite[Proposition 6.6]{Ben03}, the subset $B^t(S_{g,m})(\mathbf{L})$ is comparable to $C^t(S_{g,m})(\mathbf{L})$. 
\item By \cite[Lemma 4.9]{Ben03}, we have
\[C^t(S_{g,m})(\mathbf{L})\subset D^{\alpha(t)}(S_{g,m})(\mathbf{L}),\;\;C^t(S_{g,m})(\mathbf{L})\subset E^{\beta(t)}(S_{g,m})(\mathbf{L})\] where
\[\alpha(t)=1+\log_2(1+t^{-1}),\;\;\beta(t)=1+\log_2(1+t).\] 
\end{enumerate}
\end{remark}

\begin{prop}
\label{proposition:bencom}
For $\mathbf{L}\in \mathbb{R}_{>0}^{2m}$, the following subsets of $\mathcal{H}(S_{g,m})(\mathbf{L})$ are comparable to each other:
\[A^t(S_{g,m})(\mathbf{L}), \;\; B^t(S_{g,m})(\mathbf{L}),\;\;C^t(S_{g,m})(\mathbf{L}), \;\;\]
\[ F^t(S_{g,m})(\mathbf{L}),\;\;G^t(S_{g,m})(\mathbf{L}), \;\; \mathcal{H}^t(S_{g,m})(\mathbf{L}),\;\;\mathcal{AH}^{t}(S_{g,m})(\mathbf{L}).\]
\end{prop}

\begin{proof}
By Remark \ref{remark:comparable} and Corollary \ref{corollary:AHH}, it is enough to show that $G^t(S_{g,m})(\mathbf{L})$ is comparable to $B^t(S_{g,m})(\mathbf{L})$. Since $\mathbf{L}\in \mathbb{R}_{>0}^{2m}$, we are working on the canonical domains $\Omega_\rho$ that are strictly convex and the $\mathbb{RP}^2$ surfaces that are part of a cocompact quotient. To show that $B^t(S_{g,m})(\mathbf{L})\subset G^{C(t)}(S_{g,m})(\mathbf{L})$ for some $C(t)$, we replace the cross ratios of the harmonic quadruplets in the proof of \cite[Proposition 3.2]{Ben03} by the triple ratios, the same argument still works.  To show that $G^t(S_{g,m})(\mathbf{L})\subset B^{c(t)}(S_{g,m})(\mathbf{L})$ for some $c(t)$, it is a direct consequence of \cite[Proposition 2.10b]{Ben03}.
\end{proof}

For further research, we make the following quantitative conjecture for the subsets in Example \ref{example:exhaustion}. 
\begin{conjecture}
The following subsets of $\mathcal{H}(S_{g,m})(\mathbf{L})$ are polynomially comparable to each other:
\[A^t(S_{g,m})(\mathbf{L}), \;\; B^t(S_{g,m})(\mathbf{L}),\;\;C^{\log t}(S_{g,m})(\mathbf{L}),\;\;D^{\frac{t}{t-1}}(S_{g,m})(\mathbf{L}),\;\;E^t(S_{g,m})(\mathbf{L}), \;\;\]
\[ F^t(S_{g,m})(\mathbf{L}),\;\;G^t(S_{g,m})(\mathbf{L}), \;\; \mathcal{H}^t(S_{g,m})(\mathbf{L}),\;\;\mathcal{AH}^{t}(S_{g,m})(\mathbf{L}).\]
\end{conjecture}

\begin{remark}
By \cite[Proposition 3.4]{BH13}, the Blaschke metric is uniformly comparable to the Hilbert metric. Thus the subset of $\mathcal{H}(S_{g,m})(\mathbf{L})$ such that the Blaschke metric canonical areas are bounded above by $t$ is comparable to $\mathcal{AH}^t(S_{g,m})(\mathbf{L})$. By Labourie \cite{Lab07} and Loftin \cite{Lof01}, the moduli space of unmarked convex $\mathbb{RP}^2$ structures can be identified with the vector bundle over the moduli space of Riemann surface where each fiber is the vector space of holomorphic cubic differentials. Then one can define another subset with respect to the norm defined on all the fibers associated to the Blaschke metric that is comparable to $\mathcal{AH}^t(S_{g,m})(\mathbf{L})$. 
\end{remark}


\section{Goldman symplectic volume form}
In this section, using the generalized Darboux coordinate system obtained in \cite{SWZ17,SZ17}, we express the Goldman symplectic volume form on $\rm{Pos}_3(S_{g,m})(\mathbf{L})$ in a natural way.  
\subsection{Atiyah--Bott--Goldman symplectic form}
Let $\mathcal{R}_{G,S}=\operatorname{Hom}(\pi_1(S),G)/G$ be the space of representations of fundamental group of closed surface $S$ into Lie group $G$.
In \cite{AB83}, Atiyah and Bott introduced a natural symplectic form $\omega$ when $G$ is compact using de Rham cohomology. Later on, Goldman \cite{G84} generalized the symplectic form $\omega$ for non-compact Lie groups using group cohomology and showed that $\omega$ is a multiple of  the Weil--Petersson symplectic form on the Teichm\"uller space of $S$. We call $\omega$ the {\em (Atiyah--Bott--)Goldman symplectic form} for short. The Goldman symplectic form has been extended to the case where the topological surface $S$ has finitely many boundary components with fixed monodromy conjugacy classes in \cite{AM95} \cite{GHJW97} and references therein, even with marked points on the boundary in \cite{FR98}. 

There is a specific natural formula for Weil--Petersson symplectic form $\omega$ on the Teichm\"uller space. Let $\mathcal{T}(S_{g,m})(L_1,\cdots,L_m)$ be the Teichm\"uller space with fixed boundary lengths. Given a pair of pants decomposition $\{\delta_1,\cdots,\delta_{3g-3+m}\}$ of $S_{g,m}$ and the transverse arcs to the pants curves, $\mathcal{T}(S_{g,m})(L_1,\cdots,L_m)$ can be parameterized by $3g-3+m$ length functions $\ell(\delta_i)$ of the pants curves, and $3g-3+m$ twist functions $\theta(\delta_i)$, called the {\em Fenchel-Nielsen coordinates}. In \cite{Wol82,Wol83}, Wolpert provided an explicit description of the Weil--Petersson symplectic form on $\mathcal{T}(S_{g,m})(L_1,\cdots,L_m)$ in terms of the Fenchel--Nielsen coordinates, called the {\em Wolpert's Magic Formula}:
\begin{equation}
\label{equation:wol}
\omega= \sum_{i=1}^{3g-3+m} d \ell(\delta_i) \wedge d \theta(\delta_i).
\end{equation}
The above formula is crucial in \cite{Mir07a} for computing of the volume of moduli space $\mathcal{M}_{g,m}(L_1,\cdots,L_m):=\mathcal{T}(S_{g,m})(L_1,\cdots,L_m)/\rm{Mod}(S_{g,m})$ of Riemann surfaces with fixed boundary lengths with respect to the Weil--Petersson symplectic form $\omega$. 

Now let us consider $\rm{Pos}_3(S_{g,m})(\mathbf{L})$
with fixed simple root lengths on the oriented boundary components $\alpha_1$, $\cdots$, $\alpha_m$. In \cite{Kim99}, using Goldman's parametrization \cite{G90}, Wolpert's Magic Formula~\eqref{equation:wol} was generalized for $\rm{Hit}_3(S_{g,0})$ where some global Darboux coordinates (including the twist parameters and the internal parameters) were corrected in \cite{CJK19}. In \cite[Corollary 8.18]{SZ17}, Wolpert's Magic Formula~\eqref{equation:wol} was generalized for $\rm{Hit}_n(S_{g,0})$ where the global Darboux coordinates are established in \cite[Section 8]{SWZ17}. Note that the Goldman's coordinates are related to Fock--Goncharov's coordinates in \cite{BK18}. Both of these two generalizations of Wolpert's Magic Formula also work for $\rm{Pos}_3(S_{g,m})(\mathbf{L})$. We will use the latter one instead to match up with the projective invariants that we use.

\subsection{Generalized Wolpert's Magic Formula}
We recall the generalized Wolpert's Magic Formula \cite[Corollary 8.18]{SZ17} for future use.

We specify the ideal triangulation $\mathcal{T}$ subordinate to a pants decomposition of $S_{g,m}$. Let us fix an auxiliary hyperbolic structure $\rho_h$ on $S_{g,m}$. Suppose the pairwise non-intersecting oriented simple closed geodesics $\mathcal{P}=\{\delta_1,\cdots,\delta_{3g-3+m}\}$ cut $S_{g,m}$ into $2g-2+m$ pairs of pants $\mathbb{P}=\{P_1,\cdots,P_{2g-2+m}\}$. For each pair of pants $P$ of $\mathbb{P}$, we choose the peripheral group elements $\alpha_P,\beta_P,\gamma_P$ in $\pi_1(P)$ such that $\alpha_P\gamma_P\beta_P=Id$ and $P$ lies to the right of $\alpha_P,\beta_P,\gamma_P$. The inclusion of $P$ into $S_{g,m}$ induces the inclusion of $\pi_1(P)$ into $\pi_1(S_{g,m})$, thus we can view $\alpha_P,\beta_P,\gamma_P$ as elements in $\pi_1(S_{g,m})$. Let $\gamma^+$, $\gamma^-$ be the attracting and repelling fixed points of $\gamma \in \pi_1(S_{g,m})$. The natural projection from the universal cover $\widetilde{S_{g,m}}$ to $S_{g,m}$ is denoted by $\pi$. Then $\pi\{\alpha_P^-,\beta_P^-\}$ is the simple geodesic spiralling towards $\alpha_P$ and $\beta_P$ opposite to the orientation of $\alpha_P$ and $\beta_P$ respectively. In fact, the three simple geodesics $\pi\{\alpha_P^-,\beta_P^-\}$, $\pi\{\beta_P^-,\gamma_P^-\}$ and $\pi\{\gamma_P^-,\alpha_P^-\}$ cut $P$ into two ideal triangles $\pi\{\alpha_P^-,\beta_P^-,\gamma_P^-\}$ and $\pi\{\alpha_P^-,\gamma_P^-,\gamma_P\cdot\beta_P^-\}$. The {\em ideal triangulation} $\mathcal{T}$ is 
\[\mathcal{P} \bigcup \underset{P\in \mathbb{P}}{\bigcup} \left\{ \pi\{\alpha_P^-,\beta_P^-\}, \pi\{\beta_P^-,\gamma_P^-\},\pi\{\gamma_P^-,\alpha_P^-\}\right\}.\]

\begin{figure}[ht]
\centering
\includegraphics[scale=0.35]{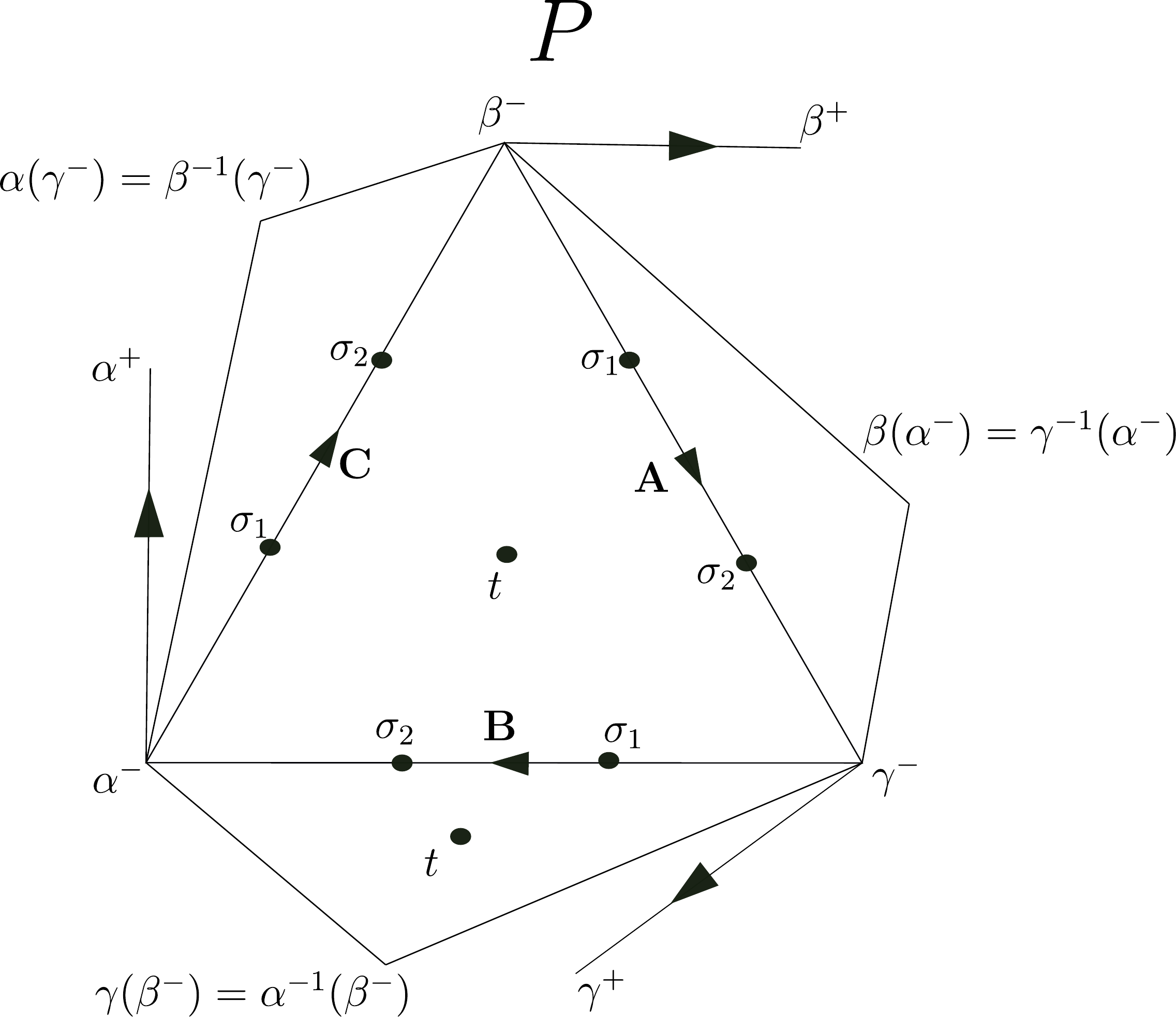}
\small
\caption{$(\gamma^-,\alpha^-,\gamma\cdot\beta^-)$ and $(\gamma^-,\beta^-,\alpha^-)$ form a lift of the pair of pants $P$ with the marking $\alpha\gamma\beta=Id$.}
\label{figure:sl3pp}
\end{figure}

Let $\mathbf{C}_P:=\pi(\alpha_P^-,\beta_P^-)$, $\mathbf{A}_P:=\pi(\beta_P^-,\gamma_P^-)$ and $\mathbf{B}_P:=\pi(\gamma_P^-,\alpha_P^-)$. Then, as in Figure \ref{figure:sl3pp}, $(\gamma_P^-,\alpha_P^-,\gamma_P\cdot\beta_P^-)$ and $(\gamma_P^-,\beta_P^-,\alpha_P^-)$ are two adjacent anticlockwise ordered ideal triangles in the universal cover with a common edge $(\gamma_P^-,\alpha_P^-)$. For any $\rho\in \rm{Pos}_3(S_{g,m})(\mathbf{L})$, there is the canonical $\rho$-equivariant map $\xi_{\rho}$.
By Definition \ref{definition:para}, for $i=1,2$,
\[D_i(\mathbf{B}_P)=D_i\left(\xi_\rho(\gamma_P^-),\xi_\rho(\alpha_P^-),\xi_\rho(\gamma_P\cdot\beta_P^-),\xi_\rho(\beta_P^-)\right).\]
Similarly, for $i=1,2$,
 \[D_i(\mathbf{C}_P)=D_i\left(\xi_\rho(\alpha_P^-),\xi_\rho(\beta_P^-),\xi_\rho(\alpha_P\cdot\gamma_P^-),\xi_\rho(\gamma_P^-)\right),\]
\[D_i(\mathbf{A}_P)=D_i\left(\xi_\rho(\beta_P^-),\xi_\rho(\gamma_P^-),\xi_\rho(\beta_P\cdot\alpha_P^-),\xi_\rho(\alpha_P^-)\right).\]
Let $\mathbf{\Delta}_P:=\pi(\alpha_P^-,\gamma_P^-,\beta_P^-)$ and $\mathbf{\Delta}_P':=\pi(\alpha_P^-,\gamma_P\cdot\beta_P^-,\gamma_P^-)$. Then 
 \[T(\mathbf{\Delta}_P)=T\left(\xi_\rho(\alpha_P^-),\xi_\rho(\gamma_P^-),\xi_\rho(\beta_P^-)\right),\]
 \[T(\mathbf{\Delta}_P')=T\left(\xi_\rho(\alpha_P^-),\xi_\rho(\gamma_P\cdot\beta_P^-),\xi_\rho(\gamma_P^-)\right).\]
\begin{notation}
For any oriented ideal edge $\mathbf{A}$, let
\[\sigma_i(\mathbf{A}):=\log \left(-D_i(\mathbf{A})\right).\]
For any anticlockwise ordered ideal triangle $\mathbf{\Delta}$, let
\[t(\mathbf{\Delta}):=\log T(\mathbf{\Delta}).\]
\end{notation}
By \cite[Proposition 13]{BH14}, we have
\begin{lem}
\label{lem:BD}
\[\ell_1(\alpha_P)=\sigma_1(\mathbf{C}_P)+\sigma_2(\mathbf{B}_P),\;\;\ell_2(\alpha_P)=\sigma_2(\mathbf{C}_P)+t(\mathbf{\Delta}_P)+\sigma_1(\mathbf{B}_P)+t(\mathbf{\Delta}_P'),\]
\[\ell_1(\beta_P)=\sigma_1(\mathbf{A}_P)+\sigma_2(\mathbf{C}_P),\;\;\ell_2(\beta_P)=\sigma_2(\mathbf{A}_P)+t(\mathbf{\Delta}_P)+\sigma_1(\mathbf{C}_P)+t(\mathbf{\Delta}_P'),\]
\[\ell_1(\gamma_P)=\sigma_1(\mathbf{B}_P)+\sigma_2(\mathbf{A}_P),\;\;\ell_2(\gamma_P)=\sigma_2(\mathbf{B}_P)+t(\mathbf{\Delta}_P)+\sigma_1(\mathbf{A}_P)+t(\mathbf{\Delta}_P').\]
\end{lem}
In \cite[Corollary 8.18]{SZ17}, the generalized Wolpert's Magic Formula of $\omega$ is composed by two parts. The first part is related to the $2g-2+m$ pairs of pants $\mathbb{P}$ that can be described by the above projective invariants. The second part is related to the $3g-3+m$ pants curves $\mathcal{P}$ where we use certain generalized length functions and certain generalized twist functions. The generalized length functions are linear combinations of $\ell_1$ and $\ell_2$. Up to scalar, the generalized twist functions are the {\em symplectic closed edge invariants} which is defined in \cite[Section 5.2]{SWZ17}, with respect to a set of transverse arcs to $\mathcal{P}$ (called the bridge system $\mathcal{J}$ there). We want to use the twist flows introduced in Section \ref{subsection:twist} instead, which can be done by a linear transformation.

\begin{thm}\cite[Corollary 8.18]{SZ17} 
\label{theorem:Darboux}
For $\rm{Pos}_3(S_{g,m})(\mathbf{L})$, let $\mathcal{T}$ be an ideal triangulation $\mathcal{T}$ subordinate to a pants decomposition $\mathcal{P}$ and a set of transverse arcs to $\mathcal{P}$ (bridge system). Let $\mathcal{P}=\{\delta_1,\cdots,\delta_{3g-3+m}\}$ be the set of disjoint oriented pants curves in the pants decomposition. Let $\mathbb{P}$ be the collection of pairs of pants. The Goldman symplectic form
\begin{equation*}
\begin{aligned}
&\omega=\sum_{P\in \mathbb{P}} d(H(\mathcal{E}_P))\wedge d(H(\mathcal{H}_P))+ \sum_{j=1}^{3g-3+m} d \ell_1(\delta_j)\wedge d\theta_1(\delta_j) + \sum_{j=1}^{3g-3+m} d\ell_2(\delta_j)\wedge d\theta_2(\delta_j)
\\&=\sum_{P\in \mathbb{P}} d(H(\mathcal{E}_P))\wedge d(H(\mathcal{H}_P))+ \sum_{j=1}^{3g-3+m} d \ell(\delta_j)\wedge d\frac{\theta_1+\theta_2}{2}(\delta_j) 
\\&+ \sum_{j=1}^{3g-3+m} d\frac{\ell_2-\ell_1}{2}(\delta_j)\wedge d(\theta_2-\theta_1)(\delta_j),
\end{aligned}
\end{equation*}
where by \cite[Theorem 8.22]{SWZ17}
\begin{equation*}
\begin{aligned}
&H(\mathcal{E}_P)
=-\frac{1}{6}(2\sigma_1(\mathbf{A}_P)+\sigma_2(\mathbf{A}_P)+2\sigma_1(\mathbf{B}_P)\
+\sigma_2(\mathbf{B}_P)+2\sigma_1(\mathbf{C}_P)+\sigma_2(\mathbf{C}_P)
\\&+3t(\mathbf{\Delta}_P)+3t(\mathbf{\Delta}_P')),
\end{aligned}
\end{equation*} 
and
\begin{equation*}
H(\mathcal{H}_P)=-t(\mathbf{\Delta}_P)+t(\mathbf{\Delta}_P')+C_P
\end{equation*} 
with $C_P$ being a linear combination of $\ell_1$ and $\ell_2$ of oriented curves in $\mathcal{P}$.
\end{thm}

\subsection{Goldman symplectic volume form}
We are well-prepared to compute the Goldman symplectic volume form on $\rm{Pos}_3(S_{g,m})(\mathbf{L})$.
\begin{prop}[Goldman symplectic volume form]
\label{proposition:gsvf}
Let $Y_P:= -t(\mathbf{\Delta}_P)+t(\mathbf{\Delta}_P')$. Let 
\[X_P:= \frac{1}{12}\left(\sigma_2(\mathbf{A}_P)-\sigma_1(\mathbf{A}_P)+\sigma_2(\mathbf{B}_P)-\sigma_1(\mathbf{B}_P)+\sigma_2(\mathbf{C}_P)-\sigma_1(\mathbf{C}_P)\right).\]
The Goldman symplectic volume form $dVol$ on $\rm{Pos}_3(S_{g,m})(\mathbf{L})$ is 
\begin{equation*}
\frac{\omega^{8g-8+3m}}{(8g-8+3m)!}=\bigwedge_{P\in \mathbb{P}} d(X_P)\wedge d(Y_P)\bigwedge_{j=1}^{3g-3+m} d\ell_1(\delta_j)\wedge d\ell_2(\delta_j)  \bigwedge_{j=1}^{3g-3+m} d(\theta_2-\theta_1)(\delta_j)\wedge d\frac{\theta_1+\theta_2}{2}(\delta_j).
\end{equation*} 
\end{prop}
\begin{proof}
By Theorem \ref{theorem:Darboux}, we have
\begin{equation}
\label{equation:gvol}
\frac{\omega^{8g-8+3m}}{(8g-8+3m)!}= \bigwedge_{P\in \mathbb{P}} d(H(\mathcal{E}_P))\wedge d(H(\mathcal{H}_P))\bigwedge_{j=1}^{3g-3+m} d\ell_1(\delta_j)\wedge d\theta_1(\delta_j)  \bigwedge_{j=1}^{3g-3+m} d\ell_2(\delta_j)\wedge d\theta_2(\delta_j).
\end{equation}
Notice the all the $d\ell_1$ and $d\ell_2$ of oriented curves in $\mathcal{P}$ appear in the antisymmetric wedge product. Thus we can replace $d(H(\mathcal{H}_P))$ in Equation~\eqref{equation:gvol} by $Y_P=H(\mathcal{H}_P)-C_P$.

By Lemma \ref{lem:BD}, we have 
\begin{equation*}
\begin{aligned}
&\ell_1(\alpha_P)+\ell_1(\beta_P)+\ell_1(\gamma_P)
\\&=\sigma_1(\mathbf{A}_P)+\sigma_2(\mathbf{A}_P)+\sigma_1(\mathbf{B}_P)+\sigma_2(\mathbf{B}_P)+\sigma_1(\mathbf{C}_P)+\sigma_2(\mathbf{C}_P)
\end{aligned}
\end{equation*}
and
\begin{equation*}
\begin{aligned}
&\ell_2(\alpha_P)+\ell_2(\beta_P)+\ell_2(\gamma_P)-\ell_1(\alpha_P)-\ell_1(\beta_P)-\ell_1(\gamma_P)
\\&=3t(\mathbf{\Delta}_P)+3t(\mathbf{\Delta}_P').
\end{aligned}
\end{equation*}

Then we can replace $d(H(\mathcal{E}_P))$ in Equation~\eqref{equation:gvol} by 
\[X_P=H(\mathcal{E}_P)+\frac{1}{12}\left(\ell_1(\alpha_P)+\ell_1(\beta_P)+\ell_1(\gamma_P)\right)+\frac{1}{6}\left(\ell_1(\alpha_P)+\ell_1(\beta_P)+\ell_1(\gamma_P)\right).\]
Moreover, we have
\[d\ell_1(\delta_j)\wedge d\theta_1(\delta_j)  \wedge d\ell_2(\delta_j)\wedge d\theta_2(\delta_j)= d\ell_1(\delta_j)\wedge d\ell_2(\delta_j)  \wedge d(\theta_2-\theta_1)(\delta_j)\wedge d\frac{\theta_1+\theta_2}{2}(\delta_j). \]
We conclude our formula. 
\end{proof}

\section{Goldman symplectic volume of $\mathcal{H}^t(S_{g,m})$}
In this section, we show that the Goldman symplectic volume of the moduli space $\mathcal{H}^t(S_{g,m})(\mathbf{L})$ of unmarked $t$-bounded positive convex $\mathbb{RP}^2$ structures is bounded above by a polynomial of $t$.

\begin{figure}[h!]
\includegraphics[scale=0.28]{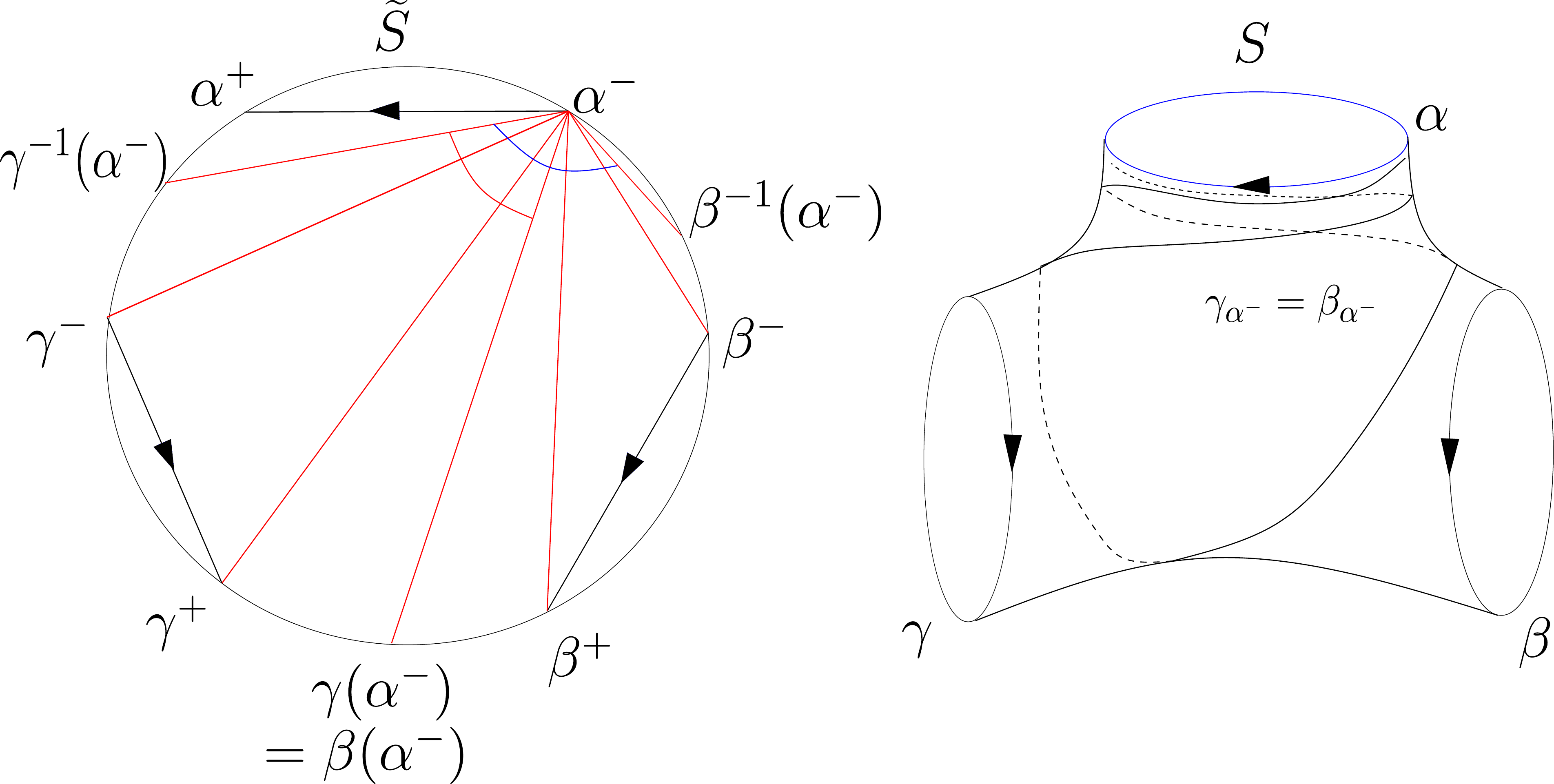}
\caption{The pair of pants $(\beta,\gamma)$ has the boundary components $\alpha$, $\beta$, $\gamma$ with $\alpha \beta^{-1} \gamma= 1$ and $(\beta,\gamma)$ is cut into $(\beta,\beta_{\alpha^-}),(\gamma,\gamma_{\alpha^-})$ along the simple bi-infinite geodesic $\gamma_{\alpha^-}=\beta_{\alpha^-}$. }
\label{Figure:pti}
\end{figure}

\subsection{Generalized McShane's identity}
Another ingredient for estimating the Goldman symplectic volume of the moduli space $\mathcal{H}^t(S_{g,m})(\mathbf{L})$ is the generalized McShane's identity \cite{HS19}, which should fit into the language of the {\em geometric recursion} \cite{ABO17}. The integration should fit into the framework of the {\em topolgical recursion} \cite{Ey14}.
\begin{thm}\cite[Generalized McShane's identity]{HS19}
\label{theorem:mcid}
For a $\operatorname{PGL}(3,\mathbb{R})$-positive representation $\rho \in \rm{Pos}_3^h(S_{g,m})$ with loxodromic boundary monodromy, let $\xi_\rho$ be the canonical $\rho$-equivariant map (Definition \ref{definition:cancho}). Let $\alpha$ be a distinguished oriented boundary component of $S_{g,m}$ such that $S_{g,m}$ is on the left side of $\alpha$. We have the equality:
\begin{equation}
\label{eq:altsummand}
\begin{aligned}
&\ell_1(\alpha)
= \sum_{(\beta,\gamma)\in \mathcal{P}_\alpha\backslash \mathcal{P}^{\partial}_\alpha}
D(\alpha,\beta,\gamma)
 + \sum_{(\beta,\gamma)\in \mathcal{P}^{\partial}_\alpha}
R(\alpha,\beta,\gamma)
\\& =\sum_{(\beta,\gamma)\in \mathcal{P}_\alpha\backslash \mathcal{P}^{\partial}_\alpha}
( \mathcal{D}(\ell_1(\alpha),\phi_1(\beta,\gamma)+\tau(\beta)+\ell_1(\beta),\phi_1(\beta,\gamma)+\tau(\gamma)+\ell_1(\gamma)) +
\\& \mathcal{D}(\ell_1(\alpha^{-1}),\phi_1(\beta^{-1},\gamma^{-1})+\tau(\beta^{-1})+\ell_1(\beta^{-1}),\phi_1(\beta^{-1},\gamma^{-1})+\tau(\gamma^{-1})+\ell_1(\gamma^{-1})) )
\\&  + \sum_{(\beta,\gamma)\in \mathcal{P}^{\partial}_\alpha}
(\mathcal{D}(\ell_1(\alpha),\phi_1'(\beta,\gamma)+\tau(\beta)+\ell_1(\beta),\phi_1'(\beta,\gamma)-\tau(\gamma^{-1})-\ell_1(\gamma^{-1}) ) +
\\& \mathcal{D}(\ell_1(\alpha^{-1}),\phi_1(\beta^{-1},\gamma^{-1})+\tau(\beta^{-1})+\ell_1(\beta^{-1}),\phi_1(\beta^{-1},\gamma^{-1})+\tau(\gamma^{-1})+\ell_1(\gamma^{-1}))  ),
\end{aligned}
\end{equation}
where $\mathcal{P}_\alpha$ is the set of the isotopy classes of pairs of pants with the boundary component $\alpha$, and $\mathcal{P}^{\partial}_\alpha$ is a subset of $\mathcal{P}_\alpha$ containing another boundary component $\gamma$ of $S_{g,m}$. For each pair of pants, we fix a marking on the boundary components $\alpha,\beta,\gamma$ such that $\alpha \beta^{-1}\gamma =1$ as in Figure \ref{Figure:pti}.
Here 
\begin{equation*}
\mathcal{D}(a,b,c):=\log \frac{e^{\frac{a}{2}}+e^{\frac{1}{2}(b+c)}}{e^{-\frac{a}{2}}+e^{\frac{1}{2}(b+c)}},
\end{equation*}
\begin{align*}
\tau(\gamma):=\log T(\alpha^-,\gamma \alpha^-, \gamma^+),\;\;\; \tau(\gamma^{-1})=-\tau(\gamma),
\end{align*} 
\begin{align*}
\phi_1(\beta,\gamma):=\log \frac{\cosh \frac{\log (-D_2(\alpha^-,\gamma(\alpha^-),\beta^+,\gamma^+))}{2}}{\cosh \frac{\log (-D_1(\alpha^-,\gamma(\alpha^-),\beta^+,\gamma^+))}{2}},
\end{align*}
\begin{align*}
\phi_1'(\beta,\gamma):=\log \frac{\cosh \frac{\log (-D_2(\alpha^-,\gamma(\alpha^-),\beta^+,\gamma^-))}{2}}{\cosh \frac{\log (-D_1(\alpha^-,\gamma(\alpha^-),\beta^+,\gamma^-))}{2}}.
\end{align*}
When $(g,m)=(1,1)$, the set $\overrightarrow{\mathcal{P}}^{\partial}_\alpha$ is empty and $\phi_1(\beta,\gamma)=0$ by computation. Let ${\overrightarrow{\mathcal{C}}_{1,1}}$ be the collection of oriented simple closed curves up to homotopy on $S_{1,1}$. Then Equation \eqref{eq:altsummand} reads
\begin{equation}
\label{eq:altsummands11}
\ell_1(\alpha)
=\sum_{\gamma\in \overrightarrow{C}_{1,1}}
\mathcal{D}(\ell_1(\alpha),\tau(\beta)+\ell_1(\beta),\tau(\gamma)+\ell_1(\gamma)) .
\end{equation}

When $\rho\in \rm{Pos}_3^u(S_{1,1})$ is a positive representation with unipotent boundary monodromy. Let $p$ be the puncture of $S_{1,1}$. Then 
\begin{align}
\label{equation:SL3S11identity}
\sum_{\gamma\in{\overrightarrow{\mathcal{C}}_{1,1}}}
\frac{1}{1+e^{\ell_1(\gamma)+\tau(\gamma)}}
= 1,
\end{align}
where $\tau(\gamma)=\log T(\tilde{p},\gamma \tilde{p}, \gamma^+)$ and $(\tilde{p},\gamma \tilde{p}, \gamma^+)$ is a lift of the ideal triangle.
\end{thm}

\subsection{Case for $S_{1,1}$}

\begin{notation}
\label{notation:vol}
The Goldman symplectic volume of the moduli space $\mathcal{H}^t(S_{g,m})(\mathbf{L})$ of unmarked $t$-bounded positive convex $\mathbb{RP}^2$ structures with fixed boundary simple root lengths $\mathbf{L}$ (Definition \ref{definition:mainob}) is denoted by $V_{g,m}^t(\mathbf{L})$.
\end{notation}

Let us start with an estimate of the {\em polylogarithm}, which is defined to be
\[Li_1(x):=-\log (1-x),\]
and for any integer $k\geq 1$
\[Li_k(x):=\int_{0}^x \frac{Li_{k-1}(t)}{t} dt.\]
\begin{lem}
\label{lemma:pl}
Let $a_k:=-Li_k(-1)$ for any integer $k\geq 2$. For any $t\geq 0$ and any integer $d\geq 2$, we have
\begin{equation}
\label{equation:poly1}
t\leq \log(1+e^t)\leq t+\log 2,
\end{equation}

\begin{equation}
\label{equation:polyd}
\frac{t^d}{d!}+\sum_{k=2}^d \frac{a_k}{(d-k)!} t^{d-k}\leq -Li_d(-e^t)\leq \frac{t^d}{d!}+\frac{\log 2}{(d-1)!}t^{d-1}+\sum_{k=2}^d \frac{a_k}{(d-k)!} t^{d-k}:=P_d(t).
\end{equation}
\end{lem}
\begin{proof}
For $t\geq0$, we have
\[e^t\leq 1+e^t\leq 2e^t.\]
Thus we obtain Inequality \eqref{equation:poly1}.

We prove Inequality \eqref{equation:polyd} by induction on $d$. For $d=2$, integrating \eqref{equation:poly1} over $t\in[0,x]$, we obtain 
\[\frac{x^2}{2}\leq -Li_2(-e^x)-a_2\leq \frac{x^2}{2}+(\log2) \cdot x.\]
Thus
\[\frac{x^2}{2}+a_2\leq -Li_2(-e^x)\leq \frac{x^2}{2}+(\log2) \cdot x +a_2
\]
for any $x\geq 0$.
Suppose Inequality \eqref{equation:polyd} is true for $d-1\geq 1$, we integrate over $t\in[0,x]$. Then we obtain
\[
\frac{x^d}{d!}+\sum_{k=2}^{d-1} \frac{a_k}{(d-k)!} x^{d-k}\leq -Li_d(-e^x)-a_d\leq \frac{x^d}{d!}+\frac{\log 2}{(d-1)!}x^{d-1}+\sum_{k=2}^{d-1} \frac{a_k}{(d-k)!} x^{d-k}.
\]
Hence
\[
\frac{x^d}{d!}+\sum_{k=2}^{d} \frac{a_k}{(d-k)!} x^{d-k}\leq -Li_d(-e^x)\leq \frac{x^d}{d!}+\frac{\log 2}{(d-1)!}x^{d-1}+\sum_{k=2}^{d} \frac{a_k}{(d-k)!} x^{d-k}
\]
for any $x\geq 0$.
\end{proof}
\begin{thm}
\label{theorem:S11uvol}
The Goldman symplectic volume $V_{1,1}^t(\mathbf{0})$ is bounded above by a positive polynomial of $t$.
\end{thm}

\begin{proof}
Using the same trick as Mirzakhani \cite[Theorem 1.2]{Mir07a} on Equation \eqref{equation:SL3S11identity} of Theorem \ref{theorem:mcid}, we have
\begin{equation}
\label{equation:S11u1}
\begin{aligned}
&V_{1,1}^t(\mathbf{0})=\int_{\mathcal{H}^t(S_{1,1})(\mathbf{0})}  1\cdot d Vol=  \int_{\mathcal{H}^t(S_{1,1})(\mathbf{0})}  \sum_{\gamma\in{\overrightarrow{\mathcal{C}}_{1,1}}}
\frac{1}{1+e^{\ell_1(\gamma)+\tau(\gamma)}}
 \cdot d Vol
\\&=  \int_{\rm{Pos}_3^t(S_{1,1})(\mathbf{0})/\rm{Stab}(\gamma)}  \left( \frac{1}{1+e^{\ell_1(\gamma)+\tau(\gamma)}}+\frac{1}{1+e^{\ell_2(\gamma)-\tau(\gamma)}}\right) \cdot d Vol ,
\end{aligned}
\end{equation}
where
\begin{equation*}
\begin{aligned}
&\rm{Pos}_3^t(S_{1,1})(\mathbf{0})/\rm{Stab}(\gamma)=
\{(X_P,Y_P,\ell_1(\gamma),\theta_1(\gamma),\ell_2(\gamma),\theta_2(\gamma))\in \rm{Pos}_3^t(S_{1,1})(\mathbf{0}) \}/\\&
(X_P,Y_P,\ell_1(\gamma),\theta_1(\gamma),\ell_2(\gamma),\theta_2(\gamma))\sim (X_P,Y_P,\ell_1(\gamma),\theta_1(\gamma)+\ell_1(\gamma),\ell_2(\gamma),\theta_2(\gamma)+\ell_2(\gamma)).
\end{aligned}
\end{equation*}
By Definition \ref{definition:bpos}, we have
\[|\tau(\gamma)|\leq t,\;\;\;|(\theta_2-\theta_1)(\gamma)|\leq t, \;\;\;|\sigma_2(\mathbf{A}_P)-\sigma_1(\mathbf{A}_P)|\leq t,\;\;\;\frac{\ell_1(\gamma)}{\ell_2(\gamma)}\leq t,\;\;\;\frac{\ell_2(\gamma)}{\ell_1(\gamma)}\leq t.\]
Thus $|Y_P|\leq 2t$ and $|X_P|\leq \frac{3t}{12}$. Using Proposition \ref{proposition:gsvf}, we continue the right hand side of Equation~\eqref{equation:S11u1} 
\begin{equation}
\label{equation:S11u2}
\begin{aligned}
&\leq \frac{6 t}{12}\cdot 4t \cdot 2t \int  \frac{\ell_1(\gamma)+\ell_2(\gamma)}{2}\left( \frac{1}{1+e^{\ell_1(\gamma)-t}}+\frac{1}{1+e^{\ell_2(\gamma)-t}}\right) d \ell_1(\gamma) d \ell_2(\gamma) 
\\&= 2t^3\cdot \int  \left( \frac{\ell_1(\gamma) +\ell_2(\gamma)}{1+e^{\ell_1(\gamma)-t}}+\frac{\ell_1(\gamma)+ \ell_2(\gamma)}{1+e^{\ell_2(\gamma)-t}}\right) d \ell_1(\gamma) d \ell_2(\gamma) 
\\&\leq 2t^3\cdot \int_{0}^{+\infty}\int_0^{t \ell_1(\gamma)}   \frac{\ell_1(\gamma) +\ell_2(\gamma)}{1+e^{\ell_1(\gamma)-t}} d \ell_1(\gamma) d \ell_2(\gamma) +
\\&2t^3\cdot\int_0^{t \ell_2(\gamma)}  \int_{0}^{+\infty} \frac{\ell_1(\gamma) + \ell_2(\gamma)}{1+e^{\ell_2(\gamma)-t}} d \ell_1(\gamma) d \ell_2(\gamma)
\\&=(2t^5+4t^4)\cdot \int_{0}^{+\infty}  \frac{x^2}{1+e^{x-t}} d x.
\end{aligned}
\end{equation}
By \cite{Le87}, the {\em complete Fermi--Dirac integral}
\begin{equation}
\label{equation:lipoly}
\begin{aligned}
& \int_{0}^{+\infty}  \frac{x^d }{1+e^{x-t}} d x= \int_{0}^{+\infty}  \sum_{k=0}^{+\infty} (-1)^k e^{t\cdot(k+1)} e^{-x\cdot(k+1)} x^d d x
\\&= \sum_{k=0}^{+\infty} (-1)^k e^{t\cdot(k+1)} \int_{0}^{+\infty}  e^{-x\cdot(k+1)} x^d d x=\Gamma(d+1) \cdot \sum_{k=0}^{+\infty} \frac{(-1)^k e^{t\cdot(k+1)}}{(k+1)^{d+1}}
\\&= -d! \cdot \sum_{k=0}^{+\infty} \frac{(-e^{t})^{k+1}    }{(k+1)^{d+1}}=-d! \cdot Li_{d+1}(-e^{t}).
\end{aligned}
\end{equation}
Taking $d=2$ and combing with Equation \eqref{equation:S11u2}, we obtain
\begin{equation*}
V_{1,1}^t(\mathbf{0}) \leq -(4 t^5+8t^4) \cdot Li_3(-e^{t}).
\end{equation*}
By Lemma \ref{lemma:pl}, for any $t\geq 0$, we have
\[-Li_3(-e^{t})\leq P_3(t),\]
where $a_k=-Li_k(-1)>0$.
Thus 
\[V_{1,1}^t(\mathbf{0}) \leq Q(t),\]
where $Q(t)$ is a positive polynomial of $t$.
\end{proof}

\begin{thm}[Main theorem]
\label{thm:main}
For $2g-2+m>0$, $m>0$ and $\mathbf{L}\in \mathbb{R}_{>0}^{2m}$, the Goldman symplectic volume $V_{g,m}^t(\mathbf{L})$ (Notation \ref{notation:vol}) is bounded above by a positive polynomial of $(t,\mathbf{L})$.
\end{thm}

\begin{proof}[Proof of Theorem \ref{thm:main} for $(g,m)=(1,1)$]
Let $\mathbf{L}=(L_1,L_2):=(\ell_1(\alpha),\ell_2(\alpha))$.
Similar to Equation \eqref{equation:S11u1}, by Equation \eqref{eq:altsummands11} of Theorem \ref{theorem:mcid}, we obtain
\begin{equation}
\label{equation:sllh2}
\begin{aligned}
&L_1 \cdot V_{1,1}^t(L_1,L_2)
\\&=  \int_{\rm{Pos}_3^t(S_{1,1})(L_1,L_2)/\rm{Stab}(\gamma)} \left(\log \frac{e^{\frac{L_1}{2}}+e^{\ell_1(\gamma)+\tau(\gamma)}}{e^{-\frac{L_1}{2}}+e^{\ell_1(\gamma)+\tau(\gamma)}}+\log \frac{e^{\frac{L_1}{2}}+e^{\ell_2(\gamma)-\tau(\gamma)}}{e^{-\frac{L_1}{2}}+e^{\ell_2(\gamma)-\tau(\gamma)}}\right) d Vol ,
\end{aligned}
\end{equation}
where
\begin{equation*}
\begin{aligned}
&\rm{Pos}_3^t(S_{1,1})(L_1,L_2)/\rm{Stab}(\gamma)=
\{(X_P,Y_P,\ell_1(\gamma),\theta_1(\gamma),\ell_2(\gamma),\theta_2(\gamma))\in \rm{Pos}_3^t(S_{1,1})(L_1,L_2) \}/\\&
(X_P,Y_P,\ell_1(\gamma),\theta_1(\gamma),\ell_2(\gamma),\theta_2(\gamma))\sim (X_P,Y_P,\ell_1(\gamma),\theta_1(\gamma)+\ell_1(\gamma),\ell_2(\gamma),\theta_2(\gamma)+\ell_2(\gamma)).
\end{aligned}
\end{equation*}

To simplify the computation, taking the following derivative
\begin{equation}
\begin{aligned}
\label{equation:s11h1}
&\frac{d}{d L_1} \left(\log \frac{e^{\frac{L_1}{2}}+e^{\ell_1(\gamma)+\tau(\gamma)}}{e^{-\frac{L_1}{2}}+e^{\ell_1(\gamma)+\tau(\gamma)}}+\log \frac{e^{\frac{L_1}{2}}+e^{\ell_2(\gamma)-\tau(\gamma)}}{e^{-\frac{L_1}{2}}+e^{\ell_2(\gamma)-\tau(\gamma)}}\right)
\\&=\frac{1}{2}\left(\frac{1}{1+e^{\ell_1(\gamma)+\tau(\gamma)-\frac{L_1}{2}}}+\frac{1}{1+e^{\ell_1(\gamma)+\tau(\gamma)+\frac{L_1}{2}}}+\frac{1}{1+e^{\ell_2(\gamma)-\tau(\gamma)-\frac{L_1}{2}}}+\frac{1}{1+e^{\ell_2(\gamma)-\tau(\gamma)+\frac{L_1}{2}}}\right).
\end{aligned}
\end{equation}

Following the same arguments as Theorem \ref{theorem:S11uvol}, by Equations \eqref{equation:sllh2} and \eqref{equation:s11h1}, for any $L_1> 0$, we obtain
\begin{equation*}
\begin{aligned}
&\frac{d}{d L_1} \left(L_1 \cdot V_{1,1}^t(L_1,L_2)\right) \leq (t^5+2t^4) \int_0^{+\infty} \left(\frac{x^2}{1+e^{x-t-\frac{L_1}{2}}} + \frac{x^2}{1+e^{x-t+\frac{L_1}{2}}}\right) d x
\\&= (2t^5+4t^4) \cdot \left( -Li_3(-e^{t+\frac{L_1}{2}}) -Li_3(-e^{t-\frac{L_1}{2}}\right)  
\\&\leq  (2t^5+4t^4) \cdot\left(P_3\left(t+\frac{L_1}{2}\right)+P_3\left(t-\frac{L_1}{2}\right) \right):=Q(t,L_1).
\end{aligned}
\end{equation*}
The polynomial $Q(t,L_1)$ is a positive polynomial of $t$ (Recall $P_4(t)$ in Formula \eqref{equation:polyd}). Thus
\begin{equation*}
\begin{aligned}
&L_1 \cdot V_{1,1}^t(L_1,L_2) =\int_{0}^{L_1}\frac{d}{d x} \left(x \cdot V_{1,1}^t(x,L_2)\right)dx \leq \int_{0}^{L_1} Q(t,x) dx
\\&=L_1 \cdot R(t,L_1).
\end{aligned}
\end{equation*}
where $R(t,L_1)$ is a positive polynomial of $t$. We conclude that $V_{1,1}^t(L_1,L_2)$ is bounded above by a positive polynomial of $t$.
\end{proof}

\subsection{Case for $S_{g,m}$}
Firstly, let us generalize Mirzakhani's integration formula that will be used to cut off the pairs of pants.
A {\em simple oriented multi-curve} is a finite sum of disjoint simple oriented closed curves with positive weights, none of whose components are peripheral. We can represent a pair of pants by a multi-curve. For any simple oriented multi-curve $\bm{\gamma}=\sum_{i=1}^k c_i \gamma_i$ and any $\rho\in \rm{Pos}_3(S_{g,m})$, suppose $f^{\bm{\gamma}}$ is a measurable function from $\rm{Pos}_3(S_{g,m})(\mathbf{L})$ to $\mathbb{R}_{\geq 0}$. We define $f_{\bm{\gamma}}$ from $\mathcal{H}(S_{g,m})(\mathbf{L})$ to $\mathbb{R}_{\geq 0}$ by
\begin{equation*}
f_{\bm{\gamma}}(\rho):=\sum_{[\bm{\alpha}] \in \rm{Mod}(S_{g,m})\cdot [\bm{\gamma}]} f^{\bm{\alpha}}(\rho).
\end{equation*}
Suppose that the simple oriented multi-curve $\bm{\gamma}$ decomposes $\rho \in \mathcal{H}(S_{g,m})(\mathbf{L})$ into $s$ connected component $\rho_1,\cdots, \rho_s$ such that, for $i=1,\cdots,s$,
\begin{itemize}
\item $\rho_i\in \rm{Pos}_3(S_{g_i,m_i})$, and
\item simple root lengths of $m_i$ oriented boundary components are given by $\mathbf{L_i} \in \mathbb{R}_{> 0}^{2 m_i}$.
\end{itemize}

Using the same argument as \cite[Theorem 7.1]{Mir07a} for the twist flows $\{\frac{\theta_1(\gamma_i)+\theta_2(\gamma_i)}{2}\}$, by Proposition \ref{proposition:gsvf}, we have 
\begin{thm}[Mirzakhani's Integration Formula for $\mathcal{H}^t(S_{g,m})(\mathbf{L})$]
\label{theorem:mirint}
For any simple oriented multi-curve $\bm{\gamma}$ and $f^{\bm{\gamma}}:\rm{Pos}_3(S_{g,m})(\mathbf{L})\rightarrow\mathbb{R}_{\geq 0}$, 
\begin{equation*}
\begin{aligned}
&\int_{\mathcal{H}^t(S_{g,m})(\mathbf{L})} f_{\bm{\gamma}} d Vol 
= \frac{1}{2^{M(\bm{\gamma})} |Sym(\bm{\gamma})|} \cdot
\\& \int \prod_{i=1}^k \frac{\ell_1(\gamma_i)+\ell_2(\gamma_i)}{2} \cdot f^{\bm{\gamma}} \cdot \prod_{i=1}^s V_{g_i,m_i}^t(\mathbf{L_i})  \prod_{i=1}^k\left(d \ell_1(\gamma_i) d \ell_2(\gamma_i)d (\theta_2-\theta_1)(\gamma_i)\right), 
\end{aligned}
\end{equation*} 
where $M(\bm{\gamma})$ is the number of $i$ such that $\gamma_i$ separates off a $S_{1,1}$, and $Sym(\bm{\gamma}):=[\rm{Stab}(\bm{\gamma}): \cap_{i=1}^k \rm{Stab}(\gamma_i)]$.
\end{thm}
\begin{remark}
Different from the moduli space of Riemann surfaces, the volume $V_{0,3}^t(\mathbf{L})$ is not one. The space $\mathcal{H}_{0,3}(\mathbf{L})$ is parameterized by two internal parameters $X_P$ and $Y_P$. We have
\begin{equation}
\label{equation:03}
V_{0,3}^t(\mathbf{L})=\int_{\mathcal{H}^t_{0,3}(\mathbf{L})} 1\cdot  d X_P Y_P\leq  \frac{6 t}{12}\cdot 4t=2t^2.
\end{equation}
\end{remark}
\begin{proof}[Proof of Theorem \ref{thm:main}]
We prove the theorem by induction on $2g-2+m$.
Similar to \cite[Theorem 8.1]{Mir07a}, we compute $\frac{\partial}{\partial L_1} L_1 V_{g,m}^t(\mathbf{L})$ using Equation \eqref{eq:altsummand} of Theorem \ref{theorem:mcid} where $\ell_1(\alpha)=L_1$. Let 
\[\tilde{D}(\alpha,\beta,\gamma):=\sum_{(\delta,\eta)\in Mod(S)\cdot(\beta,\gamma) } D(\alpha,\delta,\eta)\]
and 
\[\tilde{R}(\alpha,\beta,\gamma):=\sum_{(\delta,\eta)\in Mod(S)\cdot(\beta,\gamma) } R(\alpha,\delta,\eta).\]
Recall $\mathcal{P}_\alpha$ is the set of the isotopy classes of pairs of pants with the boundary component $\alpha$, and $\mathcal{P}^{\partial}_\alpha$ is a subset of $\mathcal{P}_\alpha$ containing another boundary component $\gamma$ of $S_{g,m}$. Let $\mathcal{A}_\alpha$ ($\mathcal{B}_\alpha$ resp.) be the finite mapping class group orbits of $\mathcal{P}_\alpha \backslash \mathcal{P}^{\partial}_\alpha$ ($\mathcal{P}^{\partial}_\alpha$ resp.). Then Equation \eqref{eq:altsummand} can be rewritten as
\[ L_1
= \sum_{(\beta,\gamma)\in \mathcal{A}_\alpha}
\tilde{D}(\alpha,\beta,\gamma)
 + \sum_{(\beta,\gamma)\in \mathcal{B}_\alpha}
\tilde{R}(\alpha,\beta,\gamma).\]
Thus
\begin{equation}
\label{equation:intAB}
\begin{aligned}
&\frac{\partial}{\partial L_1} L_1 V_{g,m}^t(\mathbf{L})=
\\&\sum_{(\beta,\gamma)\in \mathcal{A}_\alpha}\int_{\mathcal{H}^t(S_{g,m})(\mathbf{L})} 
\frac{\partial}{\partial L_1}\tilde{D}(\alpha,\beta,\gamma)dVol
 + \sum_{(\beta,\gamma)\in \mathcal{B}_\alpha}\int_{\mathcal{H}^t(S_{g,m})(\mathbf{L})} 
\frac{\partial}{\partial L_1}\tilde{R}(\alpha,\beta,\gamma)dVol.
\end{aligned}
\end{equation}
We compute each individual integral of the right hand side of Equation \eqref{equation:intAB}. 

\textbf{Integral for any $(\beta,\gamma)\in\mathcal{B}_\alpha$.} By Theorem \ref{theorem:mirint}, we obtain 
\begin{equation}
\label{equation:computeb}
\begin{aligned}
&\int_{\mathcal{H}^t(S_{g,m})(\mathbf{L})} 
\frac{\partial}{\partial L_1}\tilde{R}(\alpha,\beta,\gamma)dVol= 2^{-m(g,m-1)} \int \frac{\ell_1(\beta)+ \ell_2(\beta)}{2} \cdot \frac{\partial}{\partial L_1}R(\alpha,\beta,\gamma) \cdot 
\\&
 V_{g,m-1}^t(\mathbf{L}_{\alpha,\gamma}^{\beta}) \cdot V_{0,3}^t(\mathbf{L}_P)  d \ell_1(\beta) d \ell_2(\beta) d (\theta_2-\theta_1)(\beta),
\end{aligned}
\end{equation}
where $\mathbf{L}_{\alpha,\gamma}^{\beta}$ is obtained from $\mathbf{L}$ by replacing the simple root lengths of $\alpha$ and $\gamma$ by that of $\beta$, and $\mathbf{L}_P:=(\ell_1(\alpha),\ell_2(\alpha),\ell_1(\beta),\ell_2(\beta),\ell_1(\gamma),\ell_2(\gamma))$.
Let
\[H(x,y):=\frac{1}{1+e^{\frac{x+y}{2}}}+\frac{1}{1+e^{\frac{x-y}{2}}}.\]
Then
\begin{equation*}
\frac{\partial}{\partial a} \mathcal{D}(a,b,c)=\frac{1}{2} H(b+c,a).
\end{equation*}
Thus
\begin{equation}
\label{equation:dR}
\frac{\partial}{\partial L_1}R(\alpha,\beta,\gamma)=\frac{1}{2}H(\phi_1'(\beta,\gamma)+\tau(\beta)+\ell_1(\beta)+\phi_1'(\beta,\gamma)-\tau(\gamma^{-1})-\ell_1(\gamma^{-1}),L_1 ).
\end{equation}

By Definition \ref{definition:bpos}, we have
\begin{equation}
\label{equation:phi1}
|\phi_1(\beta,\gamma)|\leq \frac{3}{2}t, \;\;|\phi_1'(\beta,\gamma)|\leq \frac{3}{2}t.
\end{equation}
Equations \eqref{equation:dR} \eqref{equation:phi1} imply
\[\frac{\partial}{\partial L_1}R(\alpha,\beta,\gamma)\leq \frac{1}{2} H(\ell_1(\beta)-\ell_2(\gamma)-5t,L_1 ).\]
Plugging the above equation and Equation \eqref{equation:03} into Equation \eqref{equation:computeb}, using $|(\theta_2-\theta_1)(\beta)|\leq t$, we obtain 
\begin{equation}
\label{equation:rpoly1}
\begin{aligned}
&\int_{\mathcal{H}^t(S_{g,m})(\mathbf{L})} 
\frac{\partial}{\partial L_1}\tilde{R}(\alpha,\beta,\gamma)dVol
\leq  2^{-m(g,m-1)-1}\cdot  2 t^2 \cdot 2t \int  \frac{\ell_1(\beta)+ \ell_2(\beta)}{2} \cdot  
\\& H(\ell_1(\beta)-\ell_2(\gamma)-5t,L_1 ) \cdot 
 V_{g,m-1}^t(\mathbf{L}_{\alpha,\gamma}^{\beta})  d \ell_1(\beta) d \ell_2(\beta),
\end{aligned}
\end{equation}
where $m(g,m)=1$ if $(g,m)=(1,1)$ and $m(g,m)=0$ otherwise.
Since $2g-2+m-1<2g-2-m$, by induction, $V_{g,m-1}^t(\mathbf{L}_{\alpha,\gamma}^{\beta})$ is a positive polynomial of $(t,\mathbf{L}_{\alpha,\gamma}^{\beta})$. 
For any positive integer $i,j$, we have
\begin{equation}
\label{equation:rpoly2}
\begin{aligned}
&\int  (\ell_1(\beta))^i\cdot (\ell_2(\beta))^j \cdot  H(\ell_1(\beta)-\ell_2(\gamma)-5t,L_1 )  d \ell_1(\beta) d \ell_2(\beta)
\\& \leq \int_{0}^{+\infty} (\ell_1(\beta))^i\cdot \left(\int_{0}^{t\ell_1(\beta)} (\ell_2(\beta))^j  d \ell_2(\beta)\right) \cdot  H(\ell_1(\beta)-\ell_2(\gamma)-5t,L_1 )  d \ell_1(\beta)
\\&=\frac{t^{j+1}}{j+1} \int_{0}^{+\infty} (\ell_1(\beta))^{i+j+1} \cdot  H(\ell_1(\beta)-\ell_2(\gamma)-5t,L_1 )  d \ell_1(\beta)
\\&=\frac{2^{i+j+2} t^{j+1}}{j+1} \int_{0}^{+\infty} x^{i+j+1} \cdot  H(2x-\ell_2(\gamma)-5t,L_1 )  d x
\\&=\frac{2^{i+j+2} t^{j+1}}{j+1} \cdot (i+j+1)!\cdot(-Li_{i+j+2}(-e^{\frac{1}{2}(5t+\ell_2(\gamma)+L_1)})
\\&-Li_{i+j+2}(-e^{\frac{1}{2}(5t+\ell_2(\gamma)-L_1)})).
\end{aligned}
\end{equation}
The last equality follows Equation \eqref{equation:lipoly}. By Lemma \ref{lemma:pl}, the last term of Equation \eqref{equation:rpoly2} is a positive polynomial of $(t,\mathbf{L})$. Plugging Equation \eqref{equation:rpoly2} into Equation \eqref{equation:rpoly1} for each possible $i,j$, we conclude that $\int_{\mathcal{H}^t(S_{g,m})(\mathbf{L})} 
\frac{\partial}{\partial L_1}\tilde{R}(\alpha,\beta,\gamma)dVol$ is bounded above by a positive polynomial of $(t,\mathbf{L})$.

The finite set $\mathcal{A}_\alpha$ is split into two parts $\mathcal{A}^{con}_\alpha$ and $\mathcal{A}^{decon}_\alpha$ depending on the subsurface $S_{g,m}\backslash (\beta,\gamma)$ is connected or not for any $(\beta,\gamma)\in \mathcal{A}_\alpha$.

\textbf{Integral for any $(\beta,\gamma)\in\mathcal{A}^{con}_\alpha$.} By Theorem \ref{theorem:mirint}, we obtain 
\begin{equation}
\label{equation:Acon1}
\begin{aligned}
&\int_{\mathcal{H}^t(S_{g,m})(\mathbf{L})} 
\frac{\partial}{\partial L_1}\tilde{D}(\alpha,\beta,\gamma)dVol
\\&= 2^{-m(g-1,m+1)-1} \int \frac{\ell_1(\beta)+ \ell_2(\beta)}{2}\cdot \frac{\ell_1(\gamma)+ \ell_2(\gamma)}{2} \cdot 
 \frac{\partial}{\partial L_1} D(\alpha,\beta,\gamma) \cdot 
\\& V_{g-1,m+1}^t(\mathbf{L}^{\beta,\gamma}_{\alpha}) \cdot V_{0,3}^t(\mathbf{L}_P)  d \ell_1(\beta) d \ell_2(\beta)d \ell_1(\gamma) d \ell_2(\gamma)d (\theta_2-\theta_1)(\beta)d (\theta_2-\theta_1)(\gamma),
\end{aligned}
\end{equation}
where $\mathbf{L}^{\beta,\gamma}_{\alpha}$ and $\mathbf{L}_P$ are defined as in Equation \eqref{equation:computeb}.
Firstly, we get 
\[\frac{\partial}{\partial L_1} D(\alpha,\beta,\gamma)=\frac{1}{2}H(\phi_1(\beta,\gamma)+\tau(\beta)+\ell_1(\beta)+\phi_1(\beta,\gamma)+\tau(\gamma)+\ell_1(\gamma),L_1).\]
Thus 
\[\frac{\partial}{\partial L_1} D(\alpha,\beta,\gamma)\leq \frac{1}{2}H(\ell_1(\beta)+\ell_1(\gamma)-5t,L_1).\]
Plugging the above equation and Equation \eqref{equation:03} into Equation \eqref{equation:Acon1}, we obtain
\begin{equation}
\label{equation:Acon2}
\begin{aligned}
&\int_{\mathcal{H}^t(S_{g,m})(\mathbf{L})} 
\frac{\partial}{\partial L_1}\tilde{D}(\alpha,\beta,\gamma)dVol\leq 2^{-m(g-1,m+1)-1} 2t^2 \cdot (2t)^2 \int \frac{\ell_1(\beta)+ \ell_2(\beta)}{2}\cdot
\\& \frac{\ell_1(\gamma)+ \ell_2(\gamma)}{2} \cdot 
 H(\ell_1(\beta)+\ell_1(\gamma)-5t,L_1) \cdot 
 V_{g-1,m+1}^t(\mathbf{L}^{\beta,\gamma}_{\alpha})   d \ell_1(\beta) d \ell_2(\beta)d \ell_1(\gamma) d \ell_2(\gamma).
\end{aligned}
\end{equation}
Since $2(g-1)-2+m+1<2g-2+m$, by induction, $V_{g-1,m+1}^t(\mathbf{L}^{\beta,\gamma}_{\alpha})$ is a polynomial of $(t,\mathbf{L})$.
For any positive integer $i,j,k,l$, we have
\begin{equation}
\label{equation:Acon3}
\begin{aligned}
&\int  (\ell_1(\beta))^i\cdot (\ell_2(\beta))^j \cdot (\ell_1(\gamma))^k\cdot (\ell_2(\gamma))^l \cdot  H(\ell_1(\beta)+\ell_1(\gamma)-5t,L_1) 
\\& d \ell_1(\beta) d \ell_2(\beta) d \ell_1(\gamma) d \ell_2(\gamma)
\\& \leq \int_{0}^{+\infty} \int_{0}^{+\infty} (\ell_1(\beta))^i\cdot \left(\int_{0}^{t\ell_1(\beta)} (\ell_2(\beta))^j  d \ell_2(\beta)\right) \cdot (\ell_1(\gamma))^k\cdot 
\\&\left(\int_{0}^{t\ell_1(\gamma)} (\ell_2(\gamma))^l  d \ell_2(\gamma)\right) \cdot  H(\ell_1(\beta)+\ell_1(\gamma)-5t,L_1)  d \ell_1(\beta)   d \ell_1(\gamma) 
\\&=\frac{t^{j+l+2}}{(j+1)(l+1)} \int_{0}^{+\infty}\int_{0}^{+\infty} x^{i+j+1} y^{k+l+1} \cdot  H(x+y-5t,L_1)  d x d y
\\&=\frac{ t^{j+l+2} (i+j+1)! (k+l+1)!}{(j+1)(l+1)(i+j+k+l+3)!} \int_{0}^{+\infty} x^{i+j+k+l+3} \cdot  H(x-5t,L_1)  d x
\\&=\frac{ 2^{i+j+k+l+4} (i+j+1)! (k+l+1)!t^{j+l+2}}{(j+1)(l+1)} (-Li_{i+j+k+l+4}(-e^{\frac{1}{2}(5t-L_1)})
\\&-Li_{i+j+k+l+4}(-e^{\frac{1}{2}(5t+L_1)})).
\end{aligned}
\end{equation}
The second last line follows \cite[page 208]{Mir07a}. By Lemma \ref{lemma:pl}, the last line of Equation \eqref{equation:Acon3} is a positive polynomial of $(t,\mathbf{L})$. Plugging Equation \eqref{equation:Acon3} into Equation \eqref{equation:Acon2} for each possible $i,j,k,l$, we get $\int_{\mathcal{H}^t(S_{g,m})(\mathbf{L})} 
\frac{\partial}{\partial L_1}\tilde{D}(\alpha,\beta,\gamma)dVol$ is bounded above by a positive polynomial of $(t,\mathbf{L})$.

\textbf{Integral for any $(\beta,\gamma)\in\mathcal{A}^{decon}_\alpha$.} The surface $S_{g,m}\backslash (\beta,\gamma)$ is two connected surface $S_{g_1,m_1+1}$ and $S_{g_2,m_2+1}$ where $g_1+g_2=g$ and $m_1+m_2=m-1$. Here $\mathbf{L}=(\ell_1(\alpha),\ell_2(\alpha), \mathbf{L}_1,\mathbf{L}_2)$. Except for $\beta$ ($\gamma$ resp.), the surface $S_{g_1,m_1+1}$ ($S_{g_2,m_2}$ resp.) has simple root boundary lengths $\mathbf{L}_1$ ($\mathbf{L}_2$ resp.). By Theorem \ref{theorem:mirint}, we obtain 
\begin{equation}
\begin{aligned}
&\int_{\mathcal{H}^t(S_{g,m})(\mathbf{L})} 
\frac{\partial}{\partial L_1}\tilde{D}(\alpha,\beta,\gamma)dVol= 2^{-m(g_1,m_1+1)-m(g_2,m_2+1)-1} \int \frac{\ell_1(\beta)+ \ell_2(\beta)}{2} \cdot
\\& \frac{\ell_1(\gamma)+ \ell_2(\gamma)}{2}  
 \cdot \frac{\partial}{\partial L_1} D(\alpha,\beta,\gamma) \cdot 
 V_{g_1,m_1+1}^t(\mathbf{L}_1^{\beta}) \cdot V_{g_2,m_2+1}^t(\mathbf{L}_2^{\gamma}) \cdot V_{0,3}^t(\mathbf{L}_P)   
 \\& d \ell_1(\beta) d \ell_2(\beta)d \ell_1(\gamma) d \ell_2(\gamma) d(\theta_2-\theta_1)(\beta) d(\theta_2-\theta_1)(\gamma).
\end{aligned}
\end{equation}
By similar argument as for $(\beta,\gamma)\in\mathcal{A}^{con}_\alpha$, we obtain $\int_{\mathcal{H}^t(S_{g,m})(\mathbf{L})} 
\frac{\partial}{\partial L_1}\tilde{D}(\alpha,\beta,\gamma)dVol$ is bounded above by a positive polynomial of $(t,\mathbf{L})$. 

Finally, we conclude that $V_{g,m}^t(\mathbf{L})$ is bounded above by a positive polynomial of $(t,\mathbf{L})$.
\end{proof}

\begin{remark}
Following the above proof, the degree of the positive polynomial of $(t,\mathbf{L})$ is bounded above by $28g-28+14m$, since the increased degree is $8$ for any $(\beta,\gamma)\in\mathcal{B}_\alpha$ and the increased degree is $14$ for any $(\beta,\gamma)\in\mathcal{A}_\alpha$ by our algorithm.
\end{remark}
By the convergence of the sequence $\sum_{k=1}^{+\infty} \frac{R(k)}{e^k}$ for any polynomial $R$, we have the following corollary.
\begin{cor}
We have $\int_{\mathcal{H}(S_{g,m})(\mathbf{L})} 
e^{-t} dVol$ is finite where (Definition \ref{definition:bpos})
\[t=\max\{mT(\rho),mD(\rho),mL(\rho),mB(\rho)\}.\] 
\end{cor}

By Proposition \ref{proposition:bencom}, we get the following.
\begin{cor}
Recall Definition \ref{definition:mainob} and Example \ref{example:exhaustion}.
The Goldman symplectic volume of 
\[\mathcal{AH}^t(S_{g,m})(\mathbf{L}),\;\;A^t(S_{g,m})(\mathbf{L}), \;\; B^t(S_{g,m})(\mathbf{L}),\;\;C^t(S_{g,m})(\mathbf{L}), \;\;\]
\[ F^t(S_{g,m})(\mathbf{L}),\;\;G^t(S_{g,m})(\mathbf{L})\] are finite.
\end{cor}

\section{Geometry of $\mathcal{AH}^{t}(S)$}
Each element in the moduli space $\mathcal{AH}^{t}(S_{g,m})(\mathbf{L})$ has the canonical area bounded above by $t$. Comparing with the Fuchsian locus which has the fixed canonical area, the condition of bounded area enlightens us to show that the moduli space $\mathcal{AH}^{t}(S_{g,m})(\mathbf{L})$ is a small neighborhood of the Fuchsian locus where many similar properties for the moduli space of Riemann surfaces still hold. 

\begin{prop}[Bers' Constant]
\label{proposition:bers}
Let $S_{g,m}$ be the surface with negative Euler characteristic. Recall the Hilbert length $\ell(\gamma):=\ell_1(\gamma)+\ell_2(\gamma)$. There is a constant $B(t)$ such that for any $\rho \in \mathcal{AH}^{t}(S_{g,m})(\mathbf{L})$ where $\mathbf{L}\in \mathbb{R}_{>0}^{2m}$, there is a pants decomposition $\mathcal{P}=\{\delta_1,\cdots, \delta_{3g-3+m}\}$ of $S_{g,m}$ with $\ell(\delta_i)\leq B(t)$ for any $i=1,\cdots,3g-3+m$.
\end{prop}
\begin{proof}
For any $\rho \in \mathcal{AH}^{t}(S_{g,m})(\mathbf{L})$ where $\mathbf{L}\in \mathbb{R}_{>0}^{2m}$, let $h_\rho$ be its Hilbert metric on the canonical domain $\Omega_\rho$. Then $\ell(\gamma)$ is the translation distance of $\gamma$ with respect to $h_\rho$. We use the same argument as in \cite[Theorem 12.8]{FM11} by induction on the number of distinct disjoint simple essential closed curves on $S_{g,m}$. Except streaming line by line, the main issue left is that the injective radius of a unit ball is bounded above by a proper function of the area of the surface. The above statement is true for a Riemannian metric by \cite{Ber76}. By \cite[Proposition 3.4]{BH13}, the blaschke metric $b^\rho$ which is Riemannian is uniformly comparable to the Hilbert metric $h_\rho$. Thus for the Hilbert metric $h_\rho$, the injective radius of a unit ball is bounded above by a proper function of the area of the surface. 
\end{proof}

The Mumford's compactness theorem \cite{Mum71} allows us to cut the moduli space of Riemann surface into thick and thin parts. We prove a similar theorem for $\mathcal{AH}^{t}(S_{g,m})(\mathbf{L})$.
\begin{defn}
Given $\epsilon>0$, the {\em thick part} $\mathcal{AH}^{t}(S_{g,m})(\mathbf{L})_\epsilon$ of $\mathcal{AH}^{t}(S_{g,m})(\mathbf{L})$ with $\mathbf{L}\in \mathbb{R}_{>0}^{2m}$ is these $\rho$ satisfies
\[\ell_1^\rho(\gamma)\geq \epsilon,\;\;\ell_2^\rho(\gamma)\geq \epsilon\]
for any essential oriented closed curve $\gamma$. 
\end{defn}

\begin{thm}
\label{theorem:mum}
The thick part $\mathcal{AH}^{t}(S_{g,m})(\mathbf{L})_\epsilon$ with $\mathbf{L}\in \mathbb{R}_{>0}^{2m}$ is compact.
\end{thm}
\begin{proof}
We adjust the proof in \cite[Theorem 12.6]{FM11} to our situation. 
Recall the coordinate system subordinate to a pants decomposition $\mathcal{P}$ and transverse arcs in Proposition \ref{proposition:gsvf}: 
\begin{itemize}
\item for each pants curve, choose $\ell_1(\gamma),\ell_1(\gamma), (\frac{\theta_1+\theta_2}{2})(\gamma),(\theta_2-\theta_1)(\gamma)$;
\item for each pair of pants $P$, choose $X_P$, $Y_P$.
\end{itemize}
For any sequence $\{\rho_i\}$ in $\mathcal{AH}^{t}(S_{g,m})(\mathbf{L})_\epsilon$, let us choose the lifts $\{\tilde{\rho}_i\}$ in $\rm{Pos}(S_{g,m})(\mathbf{L})$. By Proposition \ref{proposition:trib}, the parameters $X_P$, $Y_P$ for the sequence are all bounded within a compact interval.
By Proposition \ref{proposition:bers}, for each $\tilde{\rho}_i$, there is a pants decomposition $\mathcal{P}_i$ such that $\ell_j^{\tilde{\rho}_i}(\gamma)\in [\epsilon,B(t)]$ for $j=1,2$ and any $\gamma \in \mathcal{P}_i$. Since the mapping class group orbits of all the pants decompositions of $S_{g,m}$ are finite, we can choose a subsequence $\{\tilde{\rho}_{j_i}\}$ of $\{\tilde{\rho}_i\}$ and a sequence of mapping class group elements $\{g_i\}$ such that $g_i(\mathcal{P}_{j_i})=\mathcal{P}$. Then in the above coordinate system subordinate to $\mathcal{P}$, for any $\psi_i=g_i\cdot \tilde{\rho}_{j_i}$, we have $\ell_j^{\psi_i}(\gamma)\in [\epsilon,B(t)]$ for $j=1,2$ and any $\gamma \in \mathcal{P}$. The Dehn twists along the pants curves allow us to find a sequence $\{f_i\}$ in the mapping class group such that the sequence $\{q_i=f_i\psi_i\}$ satisfies that for any $\gamma \in \mathcal{P}$, $(\frac{\theta_1^{q_i}+\theta_2^{q_i}}{2})(\gamma)$ is bounded within a compact interval. Suppose that there is a subsequence of $\{q_i\}$, still denoted by $\{q_i\}$, such that $(\theta_2^{q_i}-\theta_1^{q_i})(\gamma)$ converges to infinity for certain $\gamma\in \mathcal{P}$. Then by \cite[Theorem 3.7]{FK16} (or Proposition \ref{proposition:tbu}), the canonical area for $q_i$ converges to infinity. Contradiction. Hence for any $\gamma\in \mathcal{P}$, $\{(\theta_2^{q_i}-\theta_1^{q_i})(\gamma)\}$ also lie in a compact interval. Hence there is a subsequence of the sequence $\{\rho_i\}$ contained in a compact set.
\end{proof}

\section*{Acknowledgements}
I thank Fran\c cois Labourie for suggesting the beautiful work of Maryam Mirzakhani to me in 2010. I thank my collaborators Yi Huang, Anna Wienhard and Tengren Zhang since two crucial tools in this paper are established in our previous collaborations. I thank Yves Benoist for suggesting the reference \cite{Ben03} to me. I thank Scott Wolpert for very helpful comments. I would like to express my gratitude to IHES, National University of Singapore, University of Luxembourg for their hospitality.



\end{document}